\documentclass{amsart}

\usepackage{amsmath}
\usepackage{amssymb}
\usepackage{amsthm}
\usepackage{mathrsfs}
\usepackage[all]{xy} 
\usepackage{bm}
\usepackage{comment}
\usepackage{color}
\usepackage{manfnt}
\usepackage{indentfirst}% 最初の段落でインデント

\usepackage{breqn}
\usepackage{autobreak}

\usepackage{mathtools}
\mathtoolsset{showonlyrefs=true}%refした数式だけ番号付け(複数回コンパイル)

\newtheorem{theorem}{Theorem}[section]
\newtheorem{lemma}[theorem]{Lemma}
\newtheorem{prop}[theorem]{Proposition}

\newtheorem*{ntheorem}{Theorem}

\theoremstyle{definition}
\newtheorem{definition}[theorem]{Definition}
\newtheorem{remarks}[theorem]{Remark}
\newtheorem{example}[theorem]{Example}

\numberwithin{equation}{section}

\allowdisplaybreaks[4]

\title[$\infty$-adic and $v$-adic Multiple zeta functions]{On $\infty$-adic and $v$-adic multiple zeta functions in positive characteristic}
\author{Daichi Matsuzuki}
\email{m19044h@math.nagoya-u.ac.jp}
\address{Graduate School of Mathematics, Nagoya University, 
Furo-cho, Chikusa-ku, Nagoya, 464-8602, Japan}
\date{\today}

\begin{document}
\maketitle
\begin{abstract}
This paper pursues positive characteristic analogues of the results of Furusho, Komori, Matsumoto and Tsumura on $p$-adic multiple $L$-functions. We consider $\infty$-adic and $v$-adic multiple zeta functions concerned by Angl\`{e}s, Ngo Dac and Tavares Ribeiro.
     Our main results in this paper consist of: (1) integral expressions of special values of $\infty$-adic multiple zeta functions, (2) integral expression of $v$-adic multiple zeta functions themselves similar to those of $p$-adic multiple $L$-functions, (3) Kummer-type congruence for the special values of $v$-adic multiple zeta functions at integers (4) relationships between special values of $\infty$-adic and those of $v$-adic multipe zeta functions at negative integers and (5) orthognal properties of multiple zeta functions and multiple zeta star functions.
\end{abstract}
\tableofcontents
\section{Introduction}
In the present paper, we consider positive characteristic analogues of $p$-adic multiple $L$-functions (\cite{Furusho2017b}).

\subsection{$p$-adic multipe $L$-functions in characteristic $0$}
Kubota and Leopoldt's $p$-adic $L$-functions are defined to be the $p$-adic continuous functions on $\mathbb{Z}_p$ interpolating certain complex analytic $L$-values at negative integers. It is known that $p$-adic $L$-functions can be written as $p$-adic integrations and are related tightly to ideal class groups of a certain kind of number fields
(cf. \cite{LangBook}).
In \cite{Furusho2017b}, Furusho, Komori, Matsumoto and Tsumura constructed $p$-adic versions of multiple zeta functions called $p$-adic multiple $L$-functions.  Their method of construction is a generalization of Koblitz's integral expressions of $p$-adic $L$-functions. Their $p$-adic multiple $L$-functions are  $p$-adic multivariable functions which interpolate the special values of complex multiple zeta ($L$-) functions at negative integers in a sense. They described the special values of $p$-adic multiple $L$-functions and obtain multiple Kummer congruence among them (\cite[Theorems 2.1 and 2.10]{Furusho2017b}).

\subsection{Positive characteristic analogues}
In positive characteristic case, Goss constructed the $v$-adic zeta function (\cite{Goss1979}), which interpolates the special values of Carlitz-Goss zeta function at negative integers $v$-adically ($v$ is a finite place of the rational function field over a finite field). This function is considered to be a positive characteristic analogue of the $p$-adic $L$-functions. Goss obtained the integral expression of $v$-adic zeta function similar to those of $p$-adic $L$-functions (\cite[Theorem 7.6, Addendum Theorem 2.1]{Goss1979}). It was shown that his $v$-adic zeta function plays an important role in a positive characteristic analogue of Iwasawa theory (\cite{Angles2020a}).

Angl\`{e}s, Ngo Dac and Tavares Ribeiro introduced the multiple generalizations of Carlitz-Goss zeta functions, which can be seen as positive characteristic analogues of multiple zeta functions, by following ideas of Goss and Thakur (\cite[Corollary 6.1]{Angles2016}). They also introduced multiple versions of Goss' $v$-adic zeta function. These $v$-adic functions can be considered as positive characteristic analogues of $p$-adic multiple $L$-functions. (Though Angl\`{e}s et al. considered their functions in quite generalized setting, we only deal with the simplest ones of their functions.)

\subsection{Results}

Let us take a power $q$ of a prime number $p$. The polynomial ring $\mathbb{F}_q[\theta]$ and the field of rational function field $\mathbb{F}_q(\theta)$ over the finite field are denoted by $A$ and $k$, respectively. We call $\infty$-adic muliple zeta functions (abbreviated as $\infty$MZFs, here $\infty$ is the place corresponding to $1/\theta \in k$) for the multivariable functions introduced in \cite[Corollary 6.1]{Angles2016} since they can be seen as positive characteristic analogues of multiple zeta functions, and denote them by $\zeta_\infty(s_1,\,\dots,\,s_r)$, see Definition \ref{DefMZF} for details. %We also consider star version $\zeta_\infty(s_1,\,\dots,\,s_r)$ of $\infty$MZFs. %$\infty$MZFs are multivariable function with values in $k_\infty\coloneqq k((1/\theta))$. 
We obtain integral expression of the special values of $\infty$MZFs:
\begin{ntheorem}[{Theorem \ref{IntegralExpression}}]
    There exists $A_v$-valued measure $\mu$ on $A_v^r$ such that the special value of the $\infty$MZF at $(-m_1,\,\dots,-m_r)$ has the following integral expression for $m_1,\,\dots,m_r \in \mathbb{N}$:
    \begin{equation}
 \zeta_{\infty}(-m_1,\,\dots,-m_r)=\int_{A_v^r}x_1^{m_1}\cdots x_r^{m_r} \mu \in A.
\end{equation}
\end{ntheorem}

Here, we fix an irreducible monic polynomial $v \in A$ and write $A_v$ for the completion of $A$ at the place corresponding to $v$ with maximal ideal $\mathfrak{m}_v$. The field of fraction of $A_v$ is denoted by $k_v$. %Goss' character space at $v$ is defined by ${\mathbb{Z}_{d,\,p}}=\mathbb{Z}/(q^{\deg v}-1) \times \mathbb{Z}_p$.
Also in the $v$-adic situation, we call $v$-adic muliple zeta functions (abbreviated as $v$MZFs) for the multivariable functions introduced in \cite[Corollary 6.2]{Angles2016} since they can be seen as positive characteristic analogues of $p$-adic multiple $L$-functions, and denote them by $\zeta_v(s_1,\,\dots,\,s_r)$. These are $k_v$-valued multivariable functions on the set $X_v^r$ with $X_v\coloneqq k_v^\times \times {\mathbb{Z}_{d,\,p}}= k_v^\times \times \mathbb{Z}/(q^{\deg v}-1) \times \mathbb{Z}_p$, see Definition \ref{Def$v$MZF} for details. 

While in the case of $v$MZFs not only the special values but also the functions themselves admit integral expression:

\begin{ntheorem}[{Theorem \ref{IntegralexpressionofStar}}]
For $t_1,\,\dots,\,t_r \in {\mathbb{Z}_{d,\,p}}$ and $\sigma_1,\,\dots,\,\sigma_r \in k_v  \setminus \mathfrak{m}_v$, there exists $A_v$-valued measure $\mu^{\sigma_1,\,\dots,\,\sigma_r}$ on $A_v^r$ such that the value of the $v$MZF at $\big((\sigma_1,\, -t_1),$ $\dots,$ $(\sigma_r,\,-t_r) \big)\in X_v^r$ has the following integral expression:
\begin{equation}
\zeta_v((\sigma_1,\, -t_1),\,\dots, (\sigma_r,\,-t_r))=\int_{(A_v  \setminus \mathfrak{m}_v)^r}x_1^{t_1}\cdots x_r^{t_r} \mu^{\sigma_1,\,\dots,\,\sigma_r} \in k_v.
\end{equation}
(The measure $\mu^{1,\,\dots,\,1}$ coincides with $\mu$ in the $\infty$-adic situation.)
\end{ntheorem}

These integral expressions imply Kummer-type congruence between special values of $v$MZFs:
\begin{ntheorem}[{Theorem \ref{KummerType}}]
If we take indices $(m_1,\,\dots,\,m_r)$ and $(l_1,\,\dots,\,l_r) \in \mathbb{Z}^r$ such that $m_i \equiv l_i  \not \equiv 0 \mod (q^{\deg v}-1)q^{e\deg v}$ for all $1\leq i \leq r$ and some $e \in \mathbb{N}$, then both $\zeta_v(m_1,\,\dots,\,m_r)$ and $\zeta_v(l_1,\,\dots,\,l_r)$ are in $A_v$ and we have the following congruence between the special values of a $v$MZF at  $(m_1,\,\dots,\,m_r)$ and $(l_1,\,\dots,\,l_r)$:
\begin{equation}
\zeta_v(m_1,\,\dots,\,m_r) \equiv \zeta_v(l_1,\,\dots,\,l_r) \mod \mathfrak{m}_v^r.
\end{equation}
\end{ntheorem}

Further, we obtain results on special values of $\infty$MZFs and $v$MZFs at negative integers:

\begin{ntheorem}[{Theorem \ref{recursiveformulae}}]
Special values of $\infty$MZFs and $v$MZFs at negative integers can be written in terms of special values of $\infty$MZFs with simpler arguments.
\end{ntheorem}

As a corollary, we obtain the following:
\begin{ntheorem}[{Thoerem \ref{recursiveformulaeCor}}]
Special values of $\infty$MZFs at negative integers can be written in terms of those of $v$MZFs, and vice versa.
\end{ntheorem}

We also consider star variants $\zeta_\infty^\star(s_1,\,\dots,\,s_r)$ and $\zeta_v^\star(s_1,\,\dots,\,s_r)$ of $\infty$MZFs and $v$MZFs in \eqref{EqDefMZSF} and \eqref{EqDefvMZSF}, respectively. In the present paper, we show all the results except for Theorem \ref{recursiveformulaeCor} for star versions in the parallel way to non-star versions. 

Stimulated by the results of  \cite[Theorem 3]{Zlobin2005} and \cite[Lemma 4.2.1]{Chang2019a}, we obtain the following relations between non-star variants and star-variants of $\infty$MZFs and $v$MZFs, which we call the orthogonal property: 
\begin{ntheorem}[{Theorem \ref{orthogonality}}]
For $s_1,\,\dots,\,s_r \in X_\infty$ (see subsection \ref{2.1} for the definition), we have
\begin{equation}
    \sum_{l=0}^{r}(-1)^{l}\zeta_\infty(s_r,\,\dots,\,s_{r-l+1})\zeta_\infty^\star(s_1,\,\dots,\,s_{r-l})=0.
\end{equation}
Similarly, for $s_1,\,\dots,\,s_r \in X_v$, we have
\begin{equation}
    \sum_{l=0}^{r}(-1)^{l}\zeta_v(s_r,\,\dots,\,s_{r-l+1})\zeta_v^\star(s_1,\,\dots,\,s_{r-l})=0.
\end{equation}
\end{ntheorem}

The plan of the present paper goes as follows: Section \ref{SectionDefinitions} is on the definitions of $\infty$MZFs and $v$MZFs. The orthogonal properties (Theorem \ref{orthogonality}) are obtained in this section. In Section \ref{Construction of multiple zeta measures}, we obtain the integral expression of them (Theorem \ref{IntegralExpression})  and observe that they imply a Kummer-type congruence (Theorem \ref{KummerType}). We devote Section \ref{SectionINductiveFormulae} to recursive formulae which enable us to write down special values of $\infty$MZFs and $v$MZFs at negative integers in terms special values with simpler arguments (Theorems \ref{recursiveformulae}, and \ref{recursiveformulaev}). As a corollary, we obtain relationships between special values of $\infty$-adic and those of $v$-adic multipe zeta functions at negative integers (Theorem \ref{EqInductiveCor}).

\section{$\infty$-adic and $v$-adic multiple zeta functions}\label{SectionDefinitions}
This section is on the definitions of $\infty$MZFs and of $v$MZFs. In section \ref{2.1}, we review the definitions of Carlitz-Goss zeta function and of Goss' $v$-adic zeta function. In section \ref{2.2}, $\infty$MZFs and $v$MZFs are defined as certain infinite sums (Definitions \ref{DefMZF} and \ref{Def$v$MZF}), which appeared in \cite[\S 6]{Angles2016}.

Let us fix a power $q$ of a prime number $p$ and take indeterminate variable $\theta$. We consider the one-variable polynomial ring $A\coloneqq \mathbb{F}_q[\theta]$ over the finite field with $q$ elements to be a positive characteristic analogue of the ring $\mathbb{Z}$. The symbol $A_+$ denotes the set of monic polynomials in $A$, which is regarded as an analogue of the set $\mathbb{N}$.
The field of fractions $k\coloneqq \mathbb{F}_q(\theta)$ of $A$ can be regarded as an analogue of the field $\mathbb{Q}$.

\subsection{Single variable case} \label{2.1}

The field of Laurent series $k_\infty\coloneqq \mathbb{F}_q((1/\theta))$ is considered to be an analogue of the field $\mathbb{R}$ and the completion $\mathbb{C}_\infty$ of its algebraic closure $\overline{k}_\infty$ is considerd to be an analogue of the field $\mathbb{C}$.
Goss' character space at $\infty$ is defined by $X_\infty\coloneqq \mathbb{C}_\infty^\times \times \mathbb{Z}_p$ (\cite{Goss1983}) and, for $s=(\sigma,\,t)\in X_\infty$ and $n \in A_+$, we define 
\begin{equation}n^s\coloneqq \sigma^{\deg n} (\theta^{-\deg n} n)^t. \label{powerXinfty}
\end{equation}
We regard the set $\mathbb{Z}$ as a subset of $X_\infty$ via the inclusion $i\mapsto (\theta^i,\,i)$, then the power $n^i$ in the sense of \eqref{powerXinfty} coincides with the power in the usual sense.

An analogue of Riemann zeta function is introduced by Carlitz (the case $s \in \mathbb{N}_{\geq1}$) and Goss (in general):

\begin{definition}[{\cite{Carlitz1935,Goss1979}}] \label{Defzeta}
The \textit{Carlitz-Goss zeta function} is the function $\zeta_\infty : X_\infty \rightarrow \mathbb{C}_\infty$ defined by
\begin{equation}
    \zeta_{\infty}(s)\coloneqq\sum_{i\geq0}S_{i}(s) \label{eqrefCZV}
\end{equation}
where 
\begin{equation}
   S_i(s)\coloneqq  \sum_{\substack{n \in A_+,\\\deg n =i}} \frac{1}{n^s}.
\end{equation}
Especially, we call the value $\zeta_{\infty}(m)$ \textit{the Carlitz zeta value} if $m \in \mathbb{N}$.
\end{definition}

Let us confirm the convergence of \eqref{eqrefCZV} in the case $s \in \mathbb{Z}$. Since $\mathbb{C}_\infty$ is ultrametric, the sum \eqref{eqrefCZV} converges if $s\geq1$. In the case $s\leq0$, the following lemma implies that the sum is finite.

\begin{lemma}[{\cite[Proposition 4.1]{Goss1979}}] \label{firstlemma}
For $m\in \mathbb{N}_{\geq0}$, $S_i(-m)=0$ if $i\geq m+1$.
\end{lemma}
\begin{proof}This lemma in proven by induction. If $m=0$, we have
\begin{equation*}
S_i(0)=1\cdot(q^{i})=0.
\end{equation*}

For $m\geq1$, the assertion is shown as follows: 
\begin{align*}
S_{i}(-m)&=\sum_{\substack{\deg a =i-1\\a:\text{monic}\\b\in\mathbb{F}_q}}(\theta a+b)^m
=\sum_{\substack{\deg a =i-1\\a:\text{monic}\\b\in\mathbb{F}_q}} \sum_{l=0}^m (\theta a)^l b^{m-l} \dbinom{m}{l}\\ 
&=\sum_{\substack{\deg a =i-1\\a:\text{monic}}} \theta^m a^m \sum_{b\in\mathbb{F}_q}b^0+ \sum_{\substack{0\leq l\leq m-1\\b\in \mathbb{F}_q}}\theta^l b^{m-l}\dbinom{m}{l}\left(\sum_{\substack{\deg a =i-1\\a:\text{ monic}}}a^l\right)=0.
\end{align*}
\end{proof}
\begin{remarks}
Goss proved that the sum \eqref{eqrefCZV} converges for any $s \in X_\infty$ (\cite[Corollary 2.3.4]{Goss1983}).
It is known that this function satisfies some properties similar to those of the Riemann zeta function. For example, Carlitz zeta values $\zeta_{\infty}(s)$ with $s\in \mathbb{N},\, (q-1)|s$ can be written down in terms of analogues of Bernoulli numbers and fundamental period $2\pi i$ of exponential function (see \cite[Theorem 9.4]{Carlitz1935} for details) and they have Euler product expression.
\end{remarks}

Irreducible monic polynomials are consider to be an analogue of prime numbers. The symbol $v$ stands for a fixed irreducible monic polynomial in $A$ of degree $d>0$. %We also denote by $v$ the valuation of $k$ corresponding to $v$. 
We denote the completions of $A$ and $k$ corresponding to $v$ by $A_v$ and $k_v$, respectively. They are consider to be analogues of $\mathbb{Z}_p$ and $\mathbb{Q}_p$, respectively.  We write $\mathfrak{m}_v$ for the maximal ideals of $A$ and $A_v$ generated by $v$.
We put 
\begin{equation}
\mathbb{Z}_{d,\,p}\coloneqq \lim_{\substack{\leftarrow\\e}}\mathbb{Z}/(q^d-1)q^{de}=\mathbb{Z}/(q^d-1) \times \mathbb{Z}_p.
\end{equation}
We can define $n^t$ for $n \in A_v  \setminus \mathfrak{m}_v$ and $t \in {\mathbb{Z}_{d,\,p}}$ (\cite[Proposition 6.2]{Goss1979}).
We further put $X_v\coloneqq k_v^\times \times {\mathbb{Z}_{d,\,p}}$ and, for $s=(\sigma,\,t)\in X_v$ and $n\in A_v  \setminus \mathfrak{m}_v$, define $n^s\coloneqq \sigma^{\deg n}n^{t}$. The set $\mathbb{Z}$ can be seen as a subset of $X_v$ via the inclusion $i\mapsto (1,\,i)$.
Goss introduced $v$-adic zeta function, which interpolates the special values of Carlitz-Goss zeta function at negative integers $v$-adically and hence can be seen as a positive characteristivc analogue of $p$-adic $L$ function: 
\begin{definition}[{\cite[Addendum, Definition 1.1]{Goss1979}}]\label{Defvzeta}
The \textit{$v$-adic zeta function} $\zeta_v(s)\,(s=(\sigma,t) \in X_v)$ is the function on $X_v$ with values in $k_v$ defined by
\begin{equation}
    \zeta_v(s)\coloneqq \sum_{i \geq0}\sigma^{-i} \tilde{S}_i(t)=\sum_{i\geq0}\tilde{S}_i(s) \in k_v,\label{EqDefzvetafunction}
\end{equation}
where 
\begin{equation}
   \tilde{S}_i(s)\coloneqq  \sum_{\substack{n \in A_+\\\deg n =i\\(n,v)=1}}\frac{1}{n^s} \in k_v.
\end{equation}

\end{definition}
Since the equality
\begin{equation}
     \tilde{S}_i(s)=\begin{cases}
   S_i(s)-v^{-s}S_{i-d}(s) &(i\geq d),\\
   S_i(s) & (i<d),\\
   \end{cases}
\end{equation}
holds if $s \in \mathbb{Z}$, we have 
\begin{equation}\zeta_v(-m)=(1-v^{m})\zeta_{\infty}(-m) \in k
\end{equation}
for any $m\in \mathbb{Z}_{\geq0}$, i.e., $v$-adic zeta function interpolates the special values of Carlitz-Goss zeta function at negative integers $v$-adically and hence can be seen as a positive characteristivc analogue of $p$-adic $L$-function. Goss proved that the sum \eqref{EqDefzvetafunction} converges (\cite[Theorem 5.4.7]{Goss1983}) and has an integral expression (Theorem \ref{SingleIntegralExpression}).

\begin{remarks}
Goss considered $\infty$-adic and $v$-adic zeta functions in more general setting where the ring $A$ is allowed to be the ring of functions of arbitrary smooth curve over $\mathbb{F}_q$ regular outside a fixed rational point $\infty$ (\cite{Goss1979, Goss1983}).
\end{remarks}

\subsection{Multivariable case} \label{2.2}

We consider positive characteristic analogues of multiple zeta functions:
\begin{definition}[{\cite[\S 6]{Angles2016}}]\label{DefMZF}
For $s_1=(\sigma_1,\,t_1),\,\dots,\,s_r=(\sigma_r,\,t_r) \in X_\infty$, we define 
\begin{align}
\zeta_{\infty}(s_1,\,\dots,\,s_r)&\coloneqq \sum_{i_1>\cdots>i_r\geq 0} S_{i_1}(s_1)\cdots S_{i_r}(s_r) \in \mathbb{C}_\infty, \label{EqDefMZF} \\
\zeta_{\infty}^\star(s_1,\,\dots,\,s_r)&\coloneqq \sum_{i_1\geq\cdots\geq i_r\geq 0} S_{i_1}(s_1)\cdots S_{i_r}(s_r) \in \mathbb{C}_\infty.\label{EqDefMZSF} \\
\end{align}
We call these functions on $X_\infty^r$ \textit{$\infty$-adic multiple zeta function} (\textit{$\infty$MZF} in short) and \textit{$\infty$-adic multiple zeta star function} (\textit{$\infty$MZSF} in short), respectively.
\end{definition}

Though it follows from the arguments in \cite[Corollary 6.1]{Angles2016} in their more general setting that the series \eqref{EqDefMZF} and \eqref{EqDefMZSF} defining $\infty$MZFs and $\infty$MZSFs converge and are rigid analytic, we present proofs in order to make our paper self-contained in Appendix \ref{AppA}.

\begin{remarks}
(1) In \cite{Angles2016}, Angl\`{e}s et al. worked in more general setting where the ring $A$ is allowed to be the ring of functions of arbitrary smooth curve over $\mathbb{F}_q$ regular outside a fixed rational point $\infty$. They also considered twisting of the numerators of summands.

(2) If $r=1$, both of $\infty$MZF and $\infty$MZSF coincide with Carlitz-Goss zeta function in Definition \ref{Defzeta}.

(3) For $m_1,\,\dots ,\, m_r \in \mathbb{Z}_{\geq 0}$, $\zeta_{\infty}(-m_1,\,\dots,\,-m_r),\, \zeta_{\infty}^\star(-m_1,\,\dots,\,-m_r) \in A$ since Lemma \ref{firstlemma} implies that the sum in \eqref{EqDefMZF} and \eqref{EqDefMZSF} are finite sums of elements in $A$.

(4) The special values of these functions at positive integers coincide with Thakur's multiple zeta values and multiple zeta star values, respectively (\cite[\S 5.10]{ThakurBook}). 
\end{remarks}

Next we consider positive characteristic analogues of $p$-adic multiple $L$-functions:

\begin{definition}[{\cite[\S 6]{Angles2016}}] \label{Def$v$MZF}
For $s_1=(\sigma_1,\,t_1),\,\dots,\,s_r=(\sigma_r,\,t_r) \in X_v$, we define
\begin{align}
\zeta_v(s_1,\,\dots,\,s_r)&\coloneqq \sum_{i_1>\cdots>i_r\geq 0} \sigma_1^{-i_1} \cdots\sigma_r^{-i_r}\tilde{S}_{i_1}(t_1)\cdots \tilde{S}_{i_r}(t_r)  \label{EqDef$v$MZF}\\
&=\sum_{i_1>\cdots>i_r\geq 0} \tilde{S}_{i_1}(s_1)\cdots \tilde{S}_{i_r}(s_r) \in k_v ,\\
\zeta_v^\star (s_1,\,\dots,\,s_r)&\coloneqq \sum_{i_1\geq \cdots \geq 
i_r\geq 0} \sigma_1^{-i_1} \cdots\sigma_r^{-i_r}\tilde{S}_{i_1}(t_1)\cdots \tilde{S}_{i_r}(t_r)  \label{EqDefvMZSF}\\
&=\sum_{i_1\geq\cdots\geq i_r\geq 0} \tilde{S}_{i_1}(s_1)\cdots \tilde{S}_{i_r}(s_r) \in k_v.\\
\end{align}
We call these functions on $X_v^r$ by \textit{$v$-adic multiple zeta functions} (\textit{$v$MZFs} in short) and \textit{$v$-adic multiple zeta star functions} (\textit{$v$MZSFs} in short), respectively.
\end{definition}

Though it follows from the arguments in \cite[Corollary 6.2]{Angles2016} that the series \eqref{EqDef$v$MZF} and \eqref{EqDefvMZSF} converge and are rigid analytic, we present proofs in order to make our paper self-contained in Appendix \ref{AppA}.

\begin{remarks}\label{remarkink}
(1) In \cite{Angles2016}, Angl\`{e}s et al worked in more general setting where the ring $A$ is allowed to be the ring of functions of arbitrary smooth curve over $\mathbb{F}_q$ regular outside a fixed rational point $\infty$. They also considered twisting of the numerators of summands.

(2) In $r=1$, both of the $v$MZF and $v$MZSF coincide with $v$-adic zeta function in Definition \ref{Defvzeta}.

(3) For $m_1,\,\dots,\, m_r \in \mathbb{Z}_{\geq 0}$, we can see that the special values $\zeta_v(-m_1,\,\dots,\,-m_r)$ and $ \zeta_v^\star(-m_1,\,\dots,\,-m_r)$ are in $A$ since Lemma \ref{firstlemma} implies that the sum in \eqref{EqDef$v$MZF} and \eqref{EqDefvMZSF} are finite sums of elements in $A$. %If $m_j=0$ for some $i=1,\,\dots,\,r-1$, we have $\zeta_v(-m_1,\,\dots,\,-m_r)=0$ and $\zeta_v^\star(-m_1,\,\dots,\,-m_r)=\zeta_v^\star(-m_1,\,\dots,\,-m_{i-1})$.

(4) The values of $\zeta_v(s_1,\,\dots,\,s_r)$ at integers coincide with (interpolated) $v$-adic multiple zeta values introduced by Thakur (\cite[\S 5.10]{ThakurBook}). It should be noticed that he suggested that they can be interpolated to $v$-adic continuous functions on $\mathbb{Z}_{d,\,p}^r$. Q.~Shen also studied these function on $\mathbb{Z}_{d,\,p}^r$ (\cite{Shenthesis}). Another formulation of $v$-adic multiple zeta values was also introduced by Chang and Mishiba (\cite{Chang2021}). Their $v$-adic multiple zeta values are different from those introduced by Thakur, but it is expected that $v$-adic multiple zeta values introduced by Thakur and those introduced by Chang and Mishiba are related each other via certain linear relations.

\end{remarks}

\subsection{Orthognal propeties of MZFs}

The following theorem, which we call the orthogonal property, is the positive characteristic analogues of \cite[Theorem 3]{Zlobin2005}.

\begin{theorem}\label{orthogonality}
For $s_1,\,\dots,\,s_r \in X_\infty$, we have
\begin{equation}
    \sum_{l=0}^{r}(-1)^{l}\zeta_\infty(s_r,\,\dots,\,s_{r-l+1})\zeta_\infty^\star(s_1,\,\dots,\,s_{r-l})=0 \label{star and nonstar}
\end{equation}
(here, the summand with $l=0$ is $\zeta_\infty^\star(s_1,\,\dots,\,s_{r})$ and the term with $l=r$ is $(-1)^r\zeta_\infty(s_r,\,\dots,\,s_{1})$). Similarly, for $s_1,\,\dots,\,s_r \in X_v$, we have
\begin{equation}
    \sum_{l=0}^{r}(-1)^{l}\zeta_v(s_r,\,\dots,\,s_{r-l+1})\zeta_v^\star(s_1,\,\dots,\,s_{r-l})=0
\end{equation}
\end{theorem}

\begin{proof}
The vanishing of the summation of the $\infty$-adic functions can be confirmed as follows:
\begin{align}
    &\quad \  \zeta_\infty^\star (s_1,\,\dots,\,s_r)=\sum_{i_1\geq\cdots \geq i_r\geq0} S_{i_1}(s_1)\cdots S_{i_r}(s_r)\\
    \intertext{abbreviating $S_{i_j}(s_j)$ as $S_j$ for $1\leq j \leq r$ (each $S_j$ depends on $i_j \in \mathbb{Z}_{\geq 0}$),}
    &=\sum_{i_r=0}^\infty S_r \sum_{i_1\geq\cdots \geq i_{r-1}}S_1 \cdots S_{r-1} -\sum_{\substack{i_1\geq\cdots \geq i_{r-1}\\i_{r}>i_{r-1}}} S_1\cdots S_r\\
    &=\zeta_\infty(s_r)\zeta_\infty^\star(s_1,\,\dots,\,s_{r-1}) -\sum_{\substack{i_1\geq\cdots \geq i_{r-1}\\i_{r}>i_{r-1}}} S_1\cdots S_r\\
    &=\zeta_\infty(s_r)\zeta_\infty^\star(s_1,\,\dots,\,s_{r-1})-\sum_{i_{r}>i_{r-1}}S_{r-1}S_r \sum_{i_1\geq\cdots\geq 
    i_{r-2}}S_1\cdots S_{r-2}\\
    &\quad +\sum_{\substack{i_1\geq\cdots \geq i_{r-2}\\i_{r}>i_{r-1}>i_{r-2}}} S_{i_1}\cdots S_{i_r}\\
    &=\zeta_\infty(s_r)\zeta_\infty^\star(s_1,\,\dots,\,s_{r-1})-\zeta_\infty(s_r,\,s_{r-1})\zeta_\infty^\star(s_1,\,\dots,\,s_{r-2})\\
    &\quad +\sum_{\substack{i_1\geq\cdots \geq i_{r-2}\\i_{r}>i_{r-1}>i_{r-2}}} S_{i_1}\cdots S_{i_r}\\
    &=\zeta_\infty(s_r)\zeta_\infty^\star(s_1,\,\dots,\,s_{r-1})-\zeta_\infty(s_r,\,s_{r-1})\zeta_\infty^\star(s_1,\,\dots,\,s_{r-2})\\
    &\quad +\sum_{i_{r}>i_{r-1}>i_{r-2}}S_{r-2}S_{r-1}S_r \sum_{i_1\geq\cdots \geq i_{r-3}}S_1\cdots S_{r-3}\\
    &\quad -\sum_{\substack{i_1\geq\cdots  \geq i_{r-3}\\i_{r}>i_{r-1}>i_{r-2}>i_{r-3}}} S_{i_1}\cdots S_{i_r}\\
    &=\zeta_\infty(s_r)\zeta_\infty^\star(s_1,\,\dots,\,s_{r-1})-\zeta_\infty(s_r,\,s_{r-1})\zeta_\infty^\star(s_1,\,\dots,\,s_{r-2})\\
    &\quad +\zeta_\infty(s_r,\,s_{r-1},\,s_{r-2})\zeta_\infty^\star(s_1,\,\dots,\,s_{r-3}) -\sum_{\substack{i_1\geq\cdots \geq i_{r-3} \\i_{r}>i_{r-1}>i_{r-2}>i_{r-3}}} S_{i_1}\cdots S_{i_r}\\
    &\ \ \vdots\\
    &=\zeta_\infty(s_r)\zeta_\infty^\star(s_1,\,\dots,\,s_{r-1})-\zeta_\infty(s_r,\,s_{r-1})\zeta_\infty^\star(s_1,\,\dots,\,s_{r-2})+\cdots\\
    &\quad \cdots-(-1)^{r-2} \sum_{\substack{i_1\geq i_2\\i_r>\cdots>i_2}}S_1 \cdots S_r\\
    &=\zeta_\infty(s_r)\zeta_\infty^\star(s_1,\,\dots,\,s_{r-1})-\zeta_\infty(s_r,\,s_{r-1})\zeta_\infty^\star(s_1,\,\dots,\,s_{r-2})+\cdots\\
    &\quad \cdots-(-1)^{r-2} \sum_{i_r>\cdots>i_3}S_3 \cdots S_r \sum_{i_1\geq i_2}S_1 S_2 -(-1)^{r-1} \sum_{\substack{i_1\\i_r>\cdots>i_2}}S_1\cdots S_r\\
     &=\zeta_\infty(s_r)\zeta_\infty^\star(s_1,\,\dots,\,s_{r-1})-\zeta_\infty(s_r,\,s_{r-1})\zeta_\infty^\star(s_1,\,\dots,\,s_{r-2})+\cdots\\
    &\quad \cdots-(-1)^{r-2} \zeta_\infty(s_r,\,\dots,\,s_3)\zeta_\infty^\star(s_1,\,s_2) -(-1)^{r-1} \sum_{\substack{i_1\\i_r>\cdots>i_2}}S_1\cdots S_r\\
    &=-(-1)^r\zeta_\infty(s_r)\zeta_\infty^\star(s_1,\,\dots,\,s_{r-1})-(-1)^2\zeta_\infty(s_r,\,s_{r-1})\zeta_\infty^\star(s_1,\,\dots,\,s_{r-2})+\cdots\\
    &\quad \cdots-(-1)^{r-2} \zeta_\infty(s_r,\,\dots,\,s_3)\zeta_\infty^\star(s_1,\,s_2) -(-1)^{r-1} \sum_{i_r>\cdots>i_2}S_2\cdots S_r \sum_{i_1}S_1 \\
    &\quad - (-1)^r\sum_{i_r>\cdots>i_1}S_1\cdots S_r\\
     &=-(-1)^r\zeta_\infty(s_r)\zeta_\infty^\star(s_1,\,\dots,\,s_{r-1})-(-1)^2\zeta_\infty(s_r,\,s_{r-1})\zeta_\infty^\star(s_1,\,\dots,\,s_{r-2})+\cdots\\
    &\quad \cdots-(-1)^{r-2} \zeta_\infty(s_r,\,\dots,\,s_3)\zeta_\infty^\star(s_1,\,s_2) -(-1)^{r-1}  \zeta_\infty(s_r,\,\dots,\,s_2)\zeta_\infty^\star(s_1)\\
    &\quad - (-1)^r \zeta_\infty(s_r,\,\dots,\,s_1)\\
    &=-\sum_{l=1}^{r}(-1)^{l}\zeta_\infty(s_r,\,\dots,\,s_{r-l+1})\zeta_\infty^\star(s_1,\,\dots,\,s_{r-l}).
\end{align}
Replacing $S_j=S_{i_j}(s_j)$ by $\tilde{S}_{i_j}(s_j)$, we can obtain the formula for $v$MZFs and $v$MZSFs.
\end{proof}
\begin{remarks}
In characteristic $0$ case, it is known that the completely same relations betsween multiple zeta values and multiple zeta star values holds, see \cite{Zlobin2005} for details.
\end{remarks}
\begin{remarks}
In positive characteristic case, Chang and Mishiba obtained similar relations
\begin{equation}
    \sum_{l=1}^r (-1)^l \operatorname{Li}_{m_r,\,\dots,\,m_{r-l-1}}(z_r,\,\dots,\,z_{r-l-1})\operatorname{Li}_{m_1,\,\dots,\,m_{r-l}}^\star(z_1,\,\dots,\,z_{r-l})=0 \label{CMstarnonstar}
\end{equation}
between Carlitz multiple polylogarithms and Carlitz multiple star polylogarithms (see \cite[Lemma 4.2.1]{Chang2019a} for the definitions).
Based on the work of Anderson and Thakur in \cite{Anderson1990}, Chang showed in \cite[Theorem 5.5.2]{Chang2014} that there are $u_{m,\,0},\,\dots,\,u_{m\,e_m} \in A$ such that we have
\begin{align*}
\Gamma_{m_1}\cdots \Gamma_{m_r} \zeta(m_1,\,\dots,\,m_r)
&=\sum_{(j_1,\,\dots,\,j_r)\in \mathfrak{J}}\theta^{j_1+\cdots+j_r}\operatorname{Li}_{m_1,\,\dots,\,m_r}(u_{m_1,\,j_1},\,\dots,\,u_{m_r,\,j_r}), \\
\Gamma_{m_1}\cdots \Gamma_{m_r} \zeta^\star(m_1,\,\dots,\,m_r)
&=\sum_{(j_1,\,\dots,\,j_r)\in \mathfrak{J}}\theta^{j_1+\cdots+j_r}\operatorname{Li}_{m_1,\,\dots,\,m_r}^\star(u_{m_1,\,j_1},\,\dots,\,u_{m_r,\,j_r})
\end{align*}
for $m_1,\,\dots,\,m_r \in \mathbb{Z}_{>0}$, where $\Gamma_{m_1},\dots, \Gamma_{m_r}$ are non-zero elements of $A$ called the Carlitz gammas and $\mathfrak{J}=\mathfrak{J}_{m_1,\,\dots,m_r}$ is the product set $\{0,\,1,\,\dots,\,e_{m_1}\}\times\cdots\times\{0,\,1,\,\dots,\,e_{m_r}\}$ for some $e_{m_j} \in \mathbb{Z}_{\geq0}$. Using these equations, we can show that \eqref{star and nonstar} can be deduced from \eqref{CMstarnonstar} in the case where $s_j=m_j \in \mathbb{Z}_{>0}$ $(j=1,\,\dots,\,r)$ as follows: 

\begin{align*}
&\sum_{(j_1,\,\dots,\,j_r)\in \mathfrak{J}_{m_1,\dots,m_r}} \theta^{j_1+\cdots+j_r}\Bigg\{\sum_{l=0}^r (-1)^l \operatorname{Li}_{m_r,\,\dots,\,m_{r-l+1}}(u_{m_r,\,j_r},\,\dots,\,u_{m_{r-l+1},\,j_{r-l+1}})\\
&\qquad\qquad\qquad\qquad\qquad\qquad\qquad\qquad\qquad\cdot\operatorname{Li}_{m_1,\,\dots,\,m_{r-l}}^\star(u_{m_1,\,j_1},\,\dots,\,u_{m_{r-l},\,j_{r-1}}) \Bigg\}\\
&=\sum_{l=0}^r (-1)^l \sum_{\substack{(j_1,\,\dots,\,j_r) \\ \in \mathfrak{J}_{m_1,\dots,m_r}}} \Bigg\{\theta^{j_r+\cdots+j_{r-l+1}}\operatorname{Li}_{m_r,\,\dots,\,m_{r-l+1}}(u_{m_r,\,j_r},\,\dots,\,u_{m_{r-l+1},\,j_{r-l+1}})\\
&\qquad\qquad\qquad\qquad\qquad\qquad\qquad\cdot\theta^{j_1+\cdots+j_l}\operatorname{Li}_{m_1,\,\dots,\,m_{r-l}}^\star(u_{m_1,\,j_1},\,\dots,\,u_{m_{r-l},\,j_{r-l}}) \Bigg\}\\&=\sum_{l=0}^r (-1)^l \sum_{\substack{(j_r,\,\dots,\,j_{r-l+1})\\ \in \mathfrak{J}_{m_r,\dots,m_{r-l+1}}}}\Bigg\{\theta^{j_r+\cdots+j_{r-l+1}}\operatorname{Li}_{m_r,\,\dots,\,m_{r-l+1}}(u_{m_r,\,j_r},\dots,u_{m_{r-l+1},\,j_{r-l+1}})\Bigg\}\\
&\qquad\qquad\qquad\qquad\quad\cdot\Bigg\{\sum_{\substack{(j_1,\,\dots,\,j_l) \\ \in\mathfrak{J}_{m_1,\dots,m_l}}}\theta^{j_1+\cdots+j_l}\operatorname{Li}_{m_1,\,\dots,\,m_{r-l}}^\star(u_{m_1,\,j_1},\,\dots,\,u_{m_{r-l},\,j_{r-l}}) \Bigg\}\\
&=\sum_{l=0}^{r}(-1)^{l}\Gamma_{m_r}\cdots \Gamma_{m_{r-l+1}} \zeta_\infty(m_r,\,\dots,\,m_{r-l+1})\Gamma_{m_1}\cdots \Gamma_{m_{r-l}} \zeta_\infty^\star(m_1,\,\dots,\,m_{r-l})\\
&=\Gamma_{m_1}\cdots \Gamma_{m_r} \sum_{l=0}^{r}(-1)^{l}\zeta_\infty(m_r,\,\dots,\,m_{r-l+1})\zeta_\infty^\star(m_1,\,\dots,\,m_{r-l}).
\end{align*}

\end{remarks}

\section{Integral expressions}\label{Construction of multiple zeta measures}

In subsection \ref{Single valued case}, we review the integral expression of special values of Carlitz-Goss zeta function at negative integers and of $v$-adic zeta functions obtained by Goss (Theorem \ref{SingleIntegralExpression}). We obtain similar integral expressions of special values of $\infty$MZFs at negative integers and $v$MZFs in subsection \ref{Multivariable case} (Theorem \ref{IntegralExpression}). From these expression of $v$MZFs, we obtain some congruence between special values of $v$MZFs in subection \ref{Multiple Kummer-type congruence}.

\subsection{Single variable case}\label{Single valued case}
In this subsection, we recall Goss' construction of $v$-adic measure which gives integral expressions of special values of Carlitz-Goss zeta function at negative integers and integral expressions of $v$-adic zeta functions (\cite{Goss1979}).

\begin{definition}[{\cite[Addendum Definition 2.1]{Goss1979}}] \label{Bernoulli-Goss measure}
For $m \in \mathbb{Z}_{\geq 0}$ and $\sigma \in k_v  \setminus \mathfrak{m}_v$, we define $A_v$-valued measure $\mu_m^\sigma$ as follows:
\begin{equation}
    \mu_m^\sigma(a+\mathfrak{m}_v^e)\coloneqq \sum_{j\geq0} \sigma^{-j} \sum_{\substack{n\in A_+\\ \deg n=j\\ n \equiv a \mod \mathfrak{m}_v^e}}n^m.
\end{equation}
We simply write $\mu^{\sigma}$ for $\mu_0^{\sigma}$.
\end{definition}

\begin{remarks}
The classical \textit{$m$-th Bernoulli distribution} $B_m$ on $\mathbb{Z}_p$ is defined by
\begin{equation}
B_m(a+p^e\mathbb{Z}_p)\coloneqq (p^e)^{m-1}\mathbb{B}_m\left(\frac{a}{p^e}\right)=m\zeta_{\equiv a (p^e)}(1-m)\quad(0\leq a <p^e),
\end{equation}
where $\mathbb{B}_m(x)$ is the $m$-th Bernoulli polynomial and $\zeta_{\equiv a (p^e)}(s)$ is the partial zeta function defined by
\begin{equation*}
\zeta_{\equiv a (p^e)}(s)\coloneqq \sum_{n \equiv a \mod p^e}\frac{1}{n^s} \quad \operatorname{Re}(s)>1,
\end{equation*}
(cf. \cite[\S 12.1, Example (3)]{WashingtonBook}). Therefore Goss' measures can be seen as positive characteristic analogues of the Bernoulli distributions. Since the Bernoulli distributions are not measures but are just distributions, we have to modify them in order to consider the integration. In contrast to the classical case, we do not need such modifications in the positive characteristic case.
\end{remarks}

Goss obtained the following integral expressions of special values of Carlitz-Goss zeta function at negative integers and of $v$-adic zeta function:
\begin{theorem}[{\cite[Theorem 7.6, Addendum Theorem 2.1]{Goss1979}}]\label{SingleIntegralExpression}
(1) For $m \in \mathbb{N}_{\geq0}$, the following equation holds:
    \begin{equation}
        \zeta_{\infty}(-m)\coloneqq \int_{A_v}x^m \mu^1.
    \end{equation}
   
   (2) For $t \in {\mathbb{Z}_{d,\,p}}$ and $\sigma \in k_v  \setminus \mathfrak{m}_v$, we have the following equation:
    \begin{equation}
            \zeta_v((\sigma,\,-t))\coloneqq \int_{A_v  \setminus \mathfrak{m}_v}x^t \mu^\sigma.
    \end{equation}
\end{theorem}

\subsection{Multivariable case}\label{Multivariable case}
In this subsection, we generalize the result by Goss dealt with in the previous subsection to the multivariable setting.
\begin{definition}\label{Def multiple zeta measure}
For $\sigma_1,\,\dots,\,\sigma_r \in  k_v  \setminus \mathfrak{m}_v$, we define $A_v$-valued measures $\mu=\mu^{\sigma_1,\,\dots,\,\sigma_r}$ and $\mu^\star=\mu^{\star\,\sigma_1,\,\dots,\,\sigma_r}$ on $A_v^r$ as follows: for $e_1,\,\dots,\,e_r \in \mathbb{N}$ and $ \alpha_1,\,\dots,\,\alpha_r \in A$,
\begin{align}
\mu((\alpha_1,\,\dots,\,\alpha_r)+\underbrace{\mathfrak{m}_v^e \times \cdots\times \mathfrak{m}_v^e}_{r \text{ times}})
\coloneqq \sum_{i_1>\cdots>i_r\geq0} \sigma_1^{-i_1}\cdots\sigma_r^{-i_r}
\sum_{\substack{n_j \in A_+\, \deg n_j=i_j\\ \alpha_j \equiv n_j \mod \mathfrak{m}_v^{e_j} \\ (\text{for } r \geq j \geq 1)}}1, \label{defMZM}\\
\mu^\star((\alpha_1,\,\dots,\,\alpha_r)+\underbrace{\mathfrak{m}_v^e \times \cdots\times \mathfrak{m}_v^e}_{r \text{ times}})
\coloneqq \sum_{i_1\geq\cdots\geq i_r \geq0} \sigma_1^{-i_1}\cdots\sigma_r^{-i_r}
\sum_{\substack{n_j \in A_+\, \deg n_j=i_j\\ \alpha_j \equiv n_j \mod \mathfrak{m}_v^{e_j} \\ (\text{for } r \geq j \geq 1)}}1.
\end{align}
\end{definition}

Let us take an open set $(\alpha_1,\,\dots,\,\alpha_r)+\mathfrak{m}_v^{e_1}\times\cdots\times \mathfrak{m}_v^{e_r} \subset A_v^r$ where $\alpha_1,\,\dots,\,\alpha_r \in A$ and $e_1,\,\dots,\,e_r \in \mathbb{Z}_{\geq0}$, and let $J \in \{1,\,\dots,\,r\}$. For fixed $i_1>\cdots>i_r$, we have
\begin{equation*}
\sum_{\substack{\deg n_j =i_j\\n_j \equiv \alpha_j (\mathfrak{m}_v^{e_j})\\(\text{for } r \geq j \geq 1)}}1=\sum_{\substack{\beta \in A_v/(\mathfrak{m}_v^{e_J+1})\\ \beta \equiv \alpha_J (\mathfrak{m}_v^{e_J})}}\sum_{\substack{\deg n_j =i_j,\,(\text{for } r \geq j \geq 1)\\n_j \equiv \alpha_j (\mathfrak{m}_v^{e_j}),\,(j \neq J) \\n_J\equiv \beta (\mathfrak{m}_v^{e_J+1})}}1.
\end{equation*}
Multiplying $\sigma_1^{-i_1}\cdots\sigma_r^{-i_r}$ and summing up all $i_1>\cdots>i_r$, we have
\begin{align*}
&\mu((\alpha_1,\,\dots,\,\alpha_r)+\mathfrak{m}_v^{e_1}\times\cdots\times \mathfrak{m}_v^{e_r})\\
=&\sum_{\substack{\beta \in A_v/(\mathfrak{m}_v^{e_J+1})\\ \beta \equiv \alpha_J (\mathfrak{m}_v^{e_J})}}\mu((\alpha_1,\,\dots,\,\beta,\,\dots,\alpha_r)+\mathfrak{m}_v^{e_1}\times\cdots \times \mathfrak{m}_v^{e_J+1}\times \cdots \times \mathfrak{m}_v^{e_r}).
\end{align*}

Therefore, we have
\begin{equation}
\mu\left(\bigcup_{i\in I}U_i\right)=\sum_{i\in I}\mu(U_i),\ 
\mu^\star \left(\bigcup_{i\in I}U_i \right)=\sum_{i\in I}\mu^\star(U_i),
\end{equation}
for any pairwise disjoint finite family $\{ U_i\}_{i\in I}$ of compact open sets in $A_v^r$.

If $r=1$, the measures defined above coincides with that defined in Definition \ref{Bernoulli-Goss measure}.

\begin{prop}

If we fix $e \in \mathbb{N}$, take $\sigma_1,\,\dots,\,\sigma_r \in k_v  \setminus \mathfrak{m}_v$ and take $\alpha_1,\,\dots,\,\alpha_r \in A$ such that $\deg \alpha_1,\,\dots,\, \deg \alpha_r < \deg v^e=de$, then the volume of the compact open set $(\alpha_1,\,\dots,\,\alpha_r)+\underbrace{\mathfrak{m}_v^e \times \cdots\times \mathfrak{m}_v^e}_{r \text{ times}}$ defined by the measure $\mu$ can be calculated as follows:
\begin{enumerate}
\item \label{main terms}
If $\deg \alpha_1>\cdots> \deg \alpha_r$ and  $\alpha_1,\,\dots,\,\alpha_r \in A_+$, then 
\begin{equation}
\mu((\alpha_1,\,\dots,\,\alpha_r)+\mathfrak{m}_v^{e}\times\cdots\times \mathfrak{m}_v^{e})
=\sigma_1^{-\deg \alpha_1}\cdots \sigma_r^{-\deg \alpha_r} + \sigma_1^{-de}\sigma_2^{-\deg \alpha_2}\cdots \sigma_r^{-\deg \alpha_r}.
\end{equation}
\item\label{vanishing term1}
If a tuple $(\alpha_1,\,\dots,\,\alpha_r)$ satisfies the following conditions:
\begin{itemize}
    \item $\deg \alpha_2>\cdots> \deg \alpha_r$,
    \item $\alpha_2,\,\dots,\,\alpha_r \in A_+$ and
    \item $\deg \alpha_1 \leq \deg \alpha_2$ or $\alpha_1 \in A \setminus A_+$,
\end{itemize}
then we have 
\begin{equation}
\mu((\alpha_1,\,\dots,\,\alpha_r)+\mathfrak{m}_v^{e}\times\cdots\times \mathfrak{m}_v^{e})
=\sigma_1^{-de}\sigma_2^{-\deg \alpha_2}\cdots \sigma_r^{-\deg \alpha_r}.
\end{equation}
\item\label{vanishing term2}
Otherwise, we have 
\begin{equation}
\mu((\alpha_1,\,\dots,\,\alpha_r)+\mathfrak{m}_v^{e}\times\cdots\times \mathfrak{m}_v^{e})
=0.
\end{equation}
\end{enumerate}
\end{prop}

\begin{proof}
For $i_1,\,\dots,\,i_r \in \mathbb{N}$, let $X_{i_1,\,\dots,\,i_r}$ be the set
\begin{equation}
    \{(n_1,\,\dots,\,n_r) \in A_+^r \mid \deg n_j =i_j,\, n_j \equiv \alpha_j \mod \mathfrak{m}_v^e \text{ for }j=1,\,\dots,r\}.
\end{equation}
Now we can write
\begin{equation}
    \mu((\alpha_1,\,\dots,\,\alpha_r)+\mathfrak{m}_v^{e}\times\cdots\times \mathfrak{m}_v^{e})
    =\sum_{i_1>\cdots>i_r} \left(\#X_{i_1,\,\dots,\,i_r} \right) \sigma_1^{-i_1}\,\dots,\,\sigma_r^{-i_r}.
\end{equation}
Let us suppose $i_1>de$. If $(n_1,\,\dots,\,n_r) \in X_{i_1,\,\dots,\,i_r}$, then we can see that $(n_1+fv^e,\,\dots,\,n_r) \in X_{i_1,\,\dots,\,i_r}$ for all $f \in \mathbb{F}_q$, so we have $q|(\#X_{i_1,\,\dots,\,i_r})$. Therefore we assume $i_1\leq de$ in what follows.

In the case (1), we have
\begin{align}
    X_{\deg \alpha_1,\,\dots,\,\deg\alpha_r}&=\{ (\alpha_1,\,\dots,\,\alpha_r) \},\\
     X_{de,\,\deg\alpha_2,\dots,\,\deg\alpha_r}&=\{ (\alpha_1+v^e,\,\alpha_2,\,\dots,\,\alpha_r) \}
\end{align}
and $X_{i_1,\,\dots,\,i_r}=\emptyset$ for other $i_1>\cdots >i_r$.
In the case (2), we have
\begin{equation}
     X_{de,\,\deg\alpha_2,\dots,\,\deg\alpha_r}=\{ (\alpha_1+v^e,\,\alpha_2,\,\dots,\,\alpha_r) \}
\end{equation}
and $X_{i_1,\,\dots,\,i_r}=\emptyset$ for other  $i_1>\cdots >i_r$. In the case (3), $X_{i_1,\,\dots,\,i_r}=\emptyset$ for any $i_1>\cdots >i_r$. Hence the result follows.
\end{proof}

In the same way, we can calculate the measure $\mu^\star$ as follows:
\begin{prop}
Let us fix $e \in \mathbb{N}$, take $\sigma_1,\,\dots,\,\sigma_r \in A_v^*$ and take $\alpha_1,\,\dots,\,\alpha_r \in A$ such that $\deg \alpha_1,\,\dots,\, \deg \alpha_r < \deg v^e =de$. We say that a tuple $(\alpha_1,\,\dots,\,\alpha_r)$ satisfies the condition $C_i$ if $\deg \alpha_i \geq \cdots\geq \deg \alpha_r$ and $\alpha_i,\,\dots,\,\alpha_r \in A_+$. If $i \in \mathbb{N}$ is the smallest such that the condition $C_i$ is satisfied, then
%If a tuple $(\alpha_1,\,\dots,\,\alpha_r)$ satisfies $C_i$ and does not satisfy $C_{i-1}$, then we have
\begin{equation*}
\mu^{\star}((\alpha_1,\,\dots,\,\alpha_r)+\mathfrak{m}_v^{e}\times\cdots\times \mathfrak{m}_v^{e})=\sum_{j=i}^{r}\sigma_1^{-de}\cdots \sigma_{j-1}^{-de }\sigma_j^{-\deg \alpha_j}\cdots\sigma_r^{-\deg \alpha_r}.
\end{equation*}
If $(\alpha_1,\,\dots,\,\alpha_r)$ does not satisfy $C_r$, then $\mu^{\star}((\alpha_1,\,\dots,\,\alpha_r)+\mathfrak{m}_v^{e}\times\cdots\times \mathfrak{m}_v^{e})=0$.
\end{prop}
We can generalize Theorem \ref{SingleIntegralExpression} to multivariable setting as follows:

\begin{theorem}\label{IntegralExpression}

For $m_1,\,\dots,\,m_r \in \mathbb{Z}_{\geq0}$,  the special value $\zeta_{\infty}(-m_1,\,\dots,-m_r)$ has the following integral expression:
\begin{equation*}
 \zeta_{\infty}(-m_1,\,\dots,-m_r)=\int_{A_v^r}x_1^{m_1}\cdots x_r^{m_r} \mu^{1,\,\dots,\,1},
\end{equation*}

and the special value $\zeta_{\infty}^\star (-m_1,\,\dots,-m_r)$ has the following integral expression:
\begin{equation*}
\zeta_{\infty}^\star(-m_1,\,\dots,-m_r)=\int_{A_v^r}x_1^{m_1}\cdots x_r^{m_r} \mu^{\star\ 1,\,\dots,\,1}.
\end{equation*}

\end{theorem}

\begin{proof}
By definitions, we can calculate Riemann sum as follows:
\begin{align}
&\qquad \quad\int_{A_v^r}x_1^{m_1}\cdots x_r^{m_r} \mu^{1,\,\dots,\,1}\\
&=\lim_{e\to \infty} \left[2 \sum_{\substack{de>\deg n_1> \cdots > \deg n_r \\n_1,\,\dots,\,n_r \in A_+}}n_1^{m_1}\, \cdots\, n_r^{m_r}+\sum_{\substack{de>\deg n_2> \cdots > \deg n_{r}, \\  n_2,\,\dots,\,n_{r} \in A_+,\, de>\deg n_1,\\ \deg n_2 \geq \deg n_{1} \text{ or } n_1 \in A \setminus A_+}} n_1^{m_1}\, \cdots\, n_r^{m_r} \right]\\
&=\lim_{e\to \infty} \left[ \sum_{\substack{de>\deg n_1> \cdots > \deg n_r \\n_1,\,\dots,\,n_r \in A_+}}n_1^{m_1}\, \cdots\, n_r^{m_r}+
\sum_{\substack{de>\deg n_2> \cdots > \deg n_r, \\n_2,\,\dots,\,n_r \in A_+,\\ de>\deg n_1,\, n_1 \in A}} n_1^{m_1}\, \cdots\, n_r^{m_r} \right]\\
&=\zeta_{\infty}(-m_1,\,\dots,-m_r)+\zeta_{\infty}(-m_2,\,\dots,-m_r)\zeta_{\infty}(-m_1) \sum_{f\in\mathbb{F}_q^\times} f^{m_1}\\
&=\zeta_{\infty}(-m_1,\,\dots,-m_r).
\end{align}
To prove the last equality, we used the facts that $\zeta_{\infty}(-m_1)=0$ if $(q-1)|m_1$ and that $\sum_{f\in\mathbb{F}_q^\times} f^{m_1}=0$ if $(q-1)\nmid m_1$. 

The second assertion can be confirmed as follows:  
\begin{align}
\begin{autobreak}
\int_{A_v^r}x_1^{m_1}\cdots x_r^{m_r} \mu^{\star\ 1,\,\dots,\,1}
=\zeta_{\infty}^\star(-m_1,\,\dots,-m_r)
\ +\zeta_{\infty}^\star(-m_2,\,\dots,-s_{r})\zeta_{\infty}^\star(-m_1)\sum_{f_1\in\mathbb{F}_q^\times} f_1^{m_1}
\ +\zeta_{\infty}^\star(-m_2,\,\dots,-m_{r})\zeta_{\infty}^\star(-m_{1})\zeta_{\infty}^\star(-m_2)\sum_{f_{1}\,f_2 \in\mathbb{F}_q^\times} f_{1}^{m_1}f_2^{m_2}
\ +\cdots
\ +\zeta_{\infty}(-m_1)\cdots\zeta_{\infty}(-m_r)\sum_{f_1,\,\dots,\,f_r \in \mathbb{F}_q^\times}f_1^{m_1}\cdots f_r^{m_r}
=\zeta_\infty^\star(-m_1,\,\dots,-m_r).
\end{autobreak}
\end{align}

\end{proof}
Contrary to the case of $\infty$MZFs and $\infty$MZSFs, we can obtain integral expressions of $v$MZFs and of $v$MZSFs not just the expressions of special values:

\begin{theorem}\label{IntegralexpressionofStar}

Let $t_1,\,\dots,\,t_r \in {\mathbb{Z}_{d,\,p}}$ and $\sigma_1,\,\dots,\,\sigma_r \in A_v  \setminus \mathfrak{m}_v$. Then we have the following integral expressions of the value $\zeta_v((\sigma_1,\, -t_1),$ $\dots,$ $(\sigma_r,\,-t_r))$ of $v$MZF:
\begin{equation}
\zeta_v((\sigma_1,\, -t_1),\,\dots, (\sigma_r,\,-t_r))=\int_{(A_v  \setminus \mathfrak{m}_v)^r}x_1^{t_1}\cdots x_r^{t_r} \mu^{\sigma_1,\,\dots,\,\sigma_r}, \label{eqvintexpression}
\end{equation}

and of the value $\zeta_v^\star((\sigma_1,\, -t_1),$ $\dots,$ $(\sigma_r,\,-t_r))$of $v$MZSFs:
\begin{equation*}
\zeta_v^\star((\sigma_1,\, -t_1),\,\dots, (\sigma_r,\,-t_r))=\int_{(A_v  \setminus \mathfrak{m}_v)^r}x_1^{t_1}\cdots x_r^{t_r} \mu^{\star\ \sigma_1,\,\dots,\,\sigma_r}.
\end{equation*}

\end{theorem}

\begin{proof}
For the proof of the assertion, we can assume that $t_1,\,\dots,\,t_r \in \mathbb{N}$ since $-\mathbb{N}$ is dense in ${\mathbb{Z}_{d,\,p}}$. Then our desired equality is obtained as follows:
\begin{align}
&\int_{(A_v  \setminus \mathfrak{m}_v)^r}x_1^{t_1}\cdots x_r^{t_r} \mu^{\sigma_1,\,\dots,\,\sigma_r}\\
=&
\begin{autobreak}\lim_{e \to \infty} \left\{ \sum_{\substack{de>\deg n_1> \cdots > \deg n_r  \\n_1,\,\dots,\,n_r \in A_+\\(n_1,\,v)=\cdots=(n_r,\,v)=1}} \sigma_1^{-\deg n_1}\cdots \sigma_r^{-\deg n_r} n_1^{t_1}\, \cdots\, n_r^{t_r} \right.
+ \left. \sum_{\substack{de>\deg n_1> \cdots > \deg n_r \\n_1,\,\dots,\,n_r \in A_+\\(n_1,\,v)=\cdots=(n_r,\,v)=1}}\sigma_1^{-de}\sigma_2^{-\deg n_2}\cdots \sigma_{r}^{-\deg n_{r}} n_1^{t_1}\, \cdots\, n_r^{t_r} \right.
+ \left. \sum_{\substack{de>\deg n_2 > \cdots > \deg n_r, \\n_2,\,\dots,\,n_{r-1} \in A_+,\, \deg n_1<de \\ \deg n_2\geq \deg n_1 \text{ or } n_1 \in A \setminus A_+\\(n_1,\,v)=\cdots=(n_r,\,v)=1}}\sigma_1^{-de}\sigma_2^{-\deg n_2}\cdots \sigma_{r}^{-\deg n_{r}} n_1^{t_1}\, \cdots\, n_r^{t_r} \right\}
\end{autobreak}\\
=&\begin{autobreak}\zeta_v((\sigma_1,\, -t_1),\,\dots, (\sigma_r,\,-t_r))+\zeta_v((\sigma_2,\, -t_2),\,\dots, (\sigma_r,\,-t_r))
\cdot\lim_{e\to\infty} \left\{\sigma_1^{-de} \zeta_v(-t_1) \right\}\sum_{f\in\mathbb{F}_q^\times} f^{t_1}=\zeta_v((\sigma_1,\, -t_1),\,\dots, (\sigma_r,\,-t_r)).
\end{autobreak}
\end{align}

The integral expression of $v$MZSVs are confirmed in the same way.
\end{proof}
\subsection{Multiple Kummer-type congruence}\label{Multiple Kummer-type congruence}

We obtain some congruence in Theorem \ref{KummerType} between values of $v$MZFs by the integral expression of Theorem \ref{IntegralexpressionofStar} which are analogous to \cite[Theorem 2.10]{Furusho2017b}.% treated in Section \ref{Construction of multiple zeta measures}.

\begin{lemma}\label{Fermat type congruences}
If we take $m,\,l \in \mathbb{Z}$ such that $m \equiv l  \mod (q^d-1)q^{(e-1)d}$ for $ e\in \mathbb{N}$, then we have the congruence $x^m \equiv x^l \mod \mathfrak{m}_v^e$ for all $x \in A_v  \setminus \mathfrak{m}_v$.
\end{lemma}
This follows from the fact that $\# (A_v/\mathfrak{m}_v^e)^\times=\# (A/\mathfrak{m}_v^e)^\times=(q^d-1)q^{(e-1)d}$.	

\begin{theorem}\label{KummerType}
For indices $(m_1,\,\dots,\,m_r) ,\,(l_1,\,\dots,\,l_r) \in \mathbb{Z}^r$ with $m_i \equiv l_i   \mod (q^d-1)q^{(e-1)d}$ for all $1\leq i \leq r$, the following congruence hold:
\begin{align}
\zeta_v(m_1,\,\dots,\,m_r) \equiv \zeta_v(l_1,\,\dots,\,l_r) \mod \mathfrak{m}_v^e,\\
\zeta_v^\star(m_1,\,\dots,\,m_r) \equiv \zeta_v^\star(l_1,\,\dots,\,l_r) \mod \mathfrak{m}_v^e.
\end{align}
\end{theorem}

\begin{proof}
By Lemma \ref{Fermat type congruences}, the function $x_1^{m_1}\cdots x_r^{m_r}-x_1^{l_1}\cdots x_r^{l_r}$ on $(A_v  \setminus \mathfrak{m}_v)^r$ has values in $\mathfrak{m}_v^e$. Since $\mu$ is $A_v$-valued, we have the formula
\begin{equation*}
\int_{(A_v  \setminus \mathfrak{m}_v)^r} (x_1^{m_1}\cdots x_r^{m_r}-x_1^{l_1}\cdots x_r^{l_r})   \mu \in \mathfrak{m}_v^e,
\end{equation*}
which yields the desired congruence $\zeta_v(m_1,\,\dots,\,m_r) \equiv \zeta_v(l_1,\,\dots,\,l_r)$. The congruence for special values of the $v$MZSV is proven by the same manner.
\end{proof}
\begin{example}
(1). For $m,\,l \in \mathbb{Z}$ such that $m \equiv l  \mod (q^d-1)q^{(e-1)d}$, we have 
\begin{equation}
    \zeta_v(m)\equiv \zeta_v(l) \mod m_v^e.
\end{equation}
If we take $m_1,\,m_2,\,l_1,\,l_2 \in \mathbb{Z}$ such that $m_1 \equiv l_1  \mod (q^d-1)q^{(e-1)d}$ and $m_2 \equiv l_2  \mod (q^d-1)q^{(e-1)d}$, then we have 
\begin{equation}\zeta_v(m_1,\,m_2)\equiv \zeta_v(l_1,\,l_2)  \text{ and } \zeta_v^\star(m_1,\,m_2)\equiv \zeta_v^\star(l_1.\,l_2) \mod m_v^e. \end{equation}
\end{example}
\section{Inductive formulae for special values of multiple zeta functions at negative integers}\label{SectionINductiveFormulae}
In subsection \ref{4.1}, we obtain recursive formula for special values of $\infty$MZFs at negative integers, which enable us to write down the special values in terms of those with simpler arguments. We also obtain similar formulae for the special values of $v$MZFs in subsection \ref{4.2}. Comparing these two results, we obtain a relationship between special values of $\infty$MZFs and those of $v$MZFs in subsection \ref{4.3}.
\subsection{Inductive formulae for special values of $\infty$MZFs at negative integers}\label{4.1}

We obtain recursive formula for special values of $\infty$MZFs at negative integers, which enable us to write down the special values in terms of those with simpler arguments. Similar results for star versions are also shown.

For $d\geq1$ and for $m_1,\,\dots,\,m_r \in \mathbb{Z}$, we write

\begin{align}
\zeta_{<d}(m_1,\,\dots,\,m_r)\coloneqq \sum_{d>i_1>\cdots>i_r\geq0}S_{i_1}(m_1)\cdots S_{i_r}(m_r),\\
\zeta_{< d}^\star(m_1,\,\dots,\,m_r)\coloneqq \sum_{d > i_1 \geq \cdots \geq i_r\geq0}S_{i_1}(m_1)\cdots S_{i_r}(m_r).
\end{align}
If $r\geq d$ we put $\zeta_{<d}(m_1,\,\dots,\,m_r)\coloneqq 0$.

Using these notations, we can write down special values of $\infty$MZFs and $\infty$MZSFs at negative integers in terms of simpler ones:
\begin{theorem}\label{recursiveformulae}
Let $m_1\,\dots,\,m_r \in \mathbb{N}_{\geq0}$ and let $v$ be a monic irreducible polynomial with degree $d$. Then we have the equalities
\begin{align}
& \quad \zeta_{\infty}(-m_1,\,\dots,\,-m_r)\label{inftyInductive}
\\
&=\sum_{l=0}^{r}\Biggr\{(-1)^l \zeta_{<d}(-m_{l+1},\,\dots,\,-m_r)\sum_{\substack{0\leq a_j <m_j\\(q-1)|(m_j-a_j)\\(\text{for }1\leq j \leq l)}} v^{a_1+\cdots+a_l} \zeta_{<d}(a_1-m_1) \cdots\\
&\quad  \cdots \zeta_{<d}(a_l-m_l)\cdot \dbinom{m_1}{a_1}\cdots \dbinom{m_l}{a_l} \cdot \zeta_{\infty}(-a_1,\,\dots,\,-a_l)\Biggr \}, 
\end{align}
and
\begin{align}
& \quad \zeta_{\infty}^\star(-m_1,\,\dots,\,-m_r) \label{inftystarrec}\\
&=\sum_{l=0}^{r}\Biggr\{ (-1)^l  \zeta_{< d}^\star (-m_{l+1},\,\dots,\,-m_r)\sum_{\substack{0\leq a_j <m_j\\(q-1)|(m_j-a_j)\\(\text{for }1\leq j \leq l)}}v^{a_1+\cdots+a_l} \zeta_{<d}(a_1-m_1) \cdots\\
&\quad \cdots \zeta_{<d}(a_l-m_l)\cdot \dbinom{m_1}{a_1}\cdots \dbinom{m_l}{a_l} \cdot \zeta_{\infty}^\star(-a_1,\,\dots,\,-a_l)\Biggr \}, \label{inftystarInductive}
\end{align}

 If $m_j=0$, then the sums for $0\leq a_j <m_j$ are interpreted as empty sums. The summands with $l=0$ are $\zeta_{< d}^\star (-m_{1},\,\dots,\,-m_r)$ and $\zeta_{< d}^\star (-m_{1},\,\dots,\,-m_r)$, respectively.
\end{theorem}
The right hand sides of \eqref{inftyInductive} and \eqref{inftystarrec} do not depend on the choice of monic prime element $v$ of degree $d$.

\begin{proof}
We have
\begin{align}
&\zeta_{\infty}(-m_1,\,\dots,\,-m_r)\\
=&\sum_{l=0}^{r} \left[\sum_{ i_1> \cdots>i_l \geq d > i_{l+1}>\cdots>i_r\geq0 } \left\{\sum_{\substack{\deg n_j= i_j\\n_j \in A_+\\(\text{for } r \geq j \geq 1)}}n_1^{m_1}\cdots n_r^{m_r}\right\}\right]\\
=&\sum_{l=0}^{r} \left[\zeta_{<d}(-m_{l+1},\,\dots,\,-m_r) \sum_{i_1>\cdots>i_l\geq d} \left\{\sum_{\substack{\deg n_j= i_j\\n_j \in A_+\\(\text{for }1\leq j \leq l)}}n_1^{m_1}\cdots n_l^{m_l}\right\}\right]\\
\intertext{putting $n_j=vh_j+\alpha_j$,}
=&\sum_{l=0}^{r} \left[\zeta_{<d}(-m_{l+1},\,\dots,\,-m_r) \sum_{i_1>\cdots>i_l\geq d} \left\{\sum_{\substack{\deg h_j= i_j-d\\h_j \in A_+\\ \deg \alpha_j<d \\(\text{for }1\leq j \leq l)}}(vh_1+\alpha_1)^{m_1}\cdots (vh_l+\alpha_l)^{m_l}\right\}\right]\\
\intertext{by expanding each $(vh_j+\alpha_j)^{m_j}$,}
=&\sum_{l=0}^{r} \Biggr[\zeta_{<d}(-m_{l+1},\,\dots,\,-m_r) \sum_{i_1>\cdots>i_l\geq d} \Biggr\{\sum_{\substack{\deg h_j= i_j-d\\h_j \in A_+\\ \deg \alpha_j<d \\(\text{for }1\leq j \leq l)}}
\Biggr(\sum_{\substack{0\leq a_j\leq m_j\\(\text{for }1\leq j \leq l)}} (vh_1)^{a_1}\cdots \\ &\cdots(vh_l)^{a_l}\cdot \alpha_1^{m_1-a_1}\cdots \alpha_l^{m_l-a_l}\cdot\dbinom{m_1}{a_1}\cdots \dbinom{m_l}{a_l}
\Biggr)\Biggr\}\Biggr]\\
\intertext{since  $\sum_{\alpha_j}\alpha_j^0=\sum_{\alpha_j}1=q^{d}=0$,}
=&\sum_{l=0}^{r} \Biggr[\zeta_{<d}(-m_{l+1},\,\dots,\,-m_r) \sum_{i_1>\cdots>i_l\geq d} \Biggr\{\sum_{\substack{\deg h_j= i_j-d\\h_j \in A_+\\ \deg \alpha_j<d \\(\text{for }1\leq j \leq l)}}
\Biggr(\sum_{\substack{0\leq a_j< m_j\\(\text{for }1\leq j \leq l)}} (vh_1)^{a_1}\cdots \\ &\cdots(vh_l)^{a_l}\cdot \alpha_1^{m_1-a_1}\cdots \alpha_l^{m_l-a_l}\cdot\dbinom{m_1}{a_1}\cdots\dbinom{m_l}{a_l}
\Biggr)\Biggr\}\Biggr]\\
=&\sum_{l=0}^{r}\Biggr\{\zeta_{<d}(-m_{l+1},\,\dots,\,-m_r)\sum_{\substack{0\leq a_j <m_j\\\deg \alpha_j<d\\(\text{for }1\leq j \leq l)}}v^{a_1+\cdots+a_l} \alpha_1^{m_1-a_1}\cdots \alpha_l^{m_l-a_l}\cdot
\dbinom{m_1}{a_1}\cdots \\
&\,\cdots \dbinom{m_l}{a_l} \cdot \zeta_{\infty}(-a_1,\,\dots,\,-a_l)\Biggr \} \\
=&\sum_{l=0}^{r}\Biggr\{\zeta_{<d}(-m_{l+1},\,\dots,\,-m_r)\sum_{\substack{0\leq a_j <m_j\\f_j \in \mathbb{F}_q^\times \\(\text{for }1\leq j \leq l)}}v^{a_1+\cdots+a_l} \zeta_{<d}(a_1-m_1) \cdots \zeta_{<d}(a_l-m_l)\cdot \\ 
&\qquad \quad \cdot f_1^{m_1-a_1}\cdots f_l^{m_l-a_l} \cdot \dbinom{m_1}{a_1}\,\cdots \dbinom{m_l}{a_l} \cdot \zeta_{\infty}(-a_1,\,\dots,\,-a_l)\Biggr \} \\
=&\sum_{l=0}^{r}\Biggr\{\zeta_{<d}(-m_{l+1},\,\dots,\,-m_r)\sum_{\substack{0\leq a_j <m_j\\(q-1)|(m_j-a_j)\\(\text{for }1\leq j \leq l)}}(-1)^l v^{a_1+\cdots+a_l} \zeta_{<d}(a_1-m_1)\cdots \\
&\qquad \quad \cdots \zeta_{<d}(a_l-m_l)\cdot \dbinom{m_1}{a_1}\cdots \dbinom{m_l}{a_l} \cdot \zeta_{\infty}(-a_1,\,\dots,\,-a_l)\Biggr \},
\end{align}
here, the last equality follows from the fact that $\sum_{f\in\mathbb{F}_q^\times} f^{m}=0$ for $m$ not divisible by $(q-1)$. 
The second formula can be confirmed as follows:
\begin{align}
& \quad \ \zeta_{\infty}^\star(-m_1,\,\dots,\,-m_r)\\
&=\sum_{l=0}^{r} \left[\sum_{ i_1\geq  \cdots \geq i_l \geq d > i_{l+1}\geq \cdots \geq i_r\geq0 } \left\{\sum_{\substack{\deg n_j= i_j\\n_j \in A_+\\(\text{for } r \geq j \geq 1)}}n_1^{m_1}\cdots n_r^{m_r}\right\}\right]\\
&=\sum_{l=0}^{r} \left[\zeta_{< d}^\star (-m_{l+1},\,\dots,\,-m_r) \sum_{i_1 \geq \cdots \geq i_l\geq d} \left\{\sum_{\substack{\deg n_j= i_j\\n_j \in A_+\\(\text{for }1\leq j \leq l)}}n_1^{m_1}\cdots n_l^{m_l}\right\}\right]\\
\intertext{putting $n_j=vh_j+\alpha_j$,}
&=\sum_{l=0}^{r} \left[\zeta_{< d}^\star (-m_{l+1},\dots,-m_r) \sum_{i_1 \geq \cdots \geq i_l\geq d} \left\{\sum_{\substack{\deg h_j= i_j-d\\h_j \in A_+\\ \deg \alpha_j<d \\(\text{for }1\leq j \leq l)}}(vh_1+\alpha_1)^{m_1}\cdots (vh_l+\alpha_l)^{m_l}\right\}\right]\\
&=\sum_{l=0}^{r} \Biggr[\zeta_{< d}^\star (-m_{l+1},\,\dots,\,-m_r) \sum_{i_1 \geq \cdots \geq i_l\geq d} \Biggr\{\sum_{\substack{\deg h_j= i_j-d\\h_j \in A_+\\ \deg \alpha_j<d \\(\text{for }1\leq j \leq l)}}
\Biggr(\sum_{\substack{0\leq a_j\leq m_j\\(\text{for }1\leq j \leq l)}} (vh_1)^{a_1}\cdots \\ & \quad \cdots(vh_l)^{a_l}\cdot \alpha_1^{m_1-a_1}\cdots \alpha_l^{m_l-a_l}\cdot \dbinom{m_1}{a_1}\cdots \dbinom{m_l}{a_l}
\Biggr)\Biggr\}\Biggr]\\
&=\sum_{l=0}^{r}\Biggr\{\zeta_{< d}^\star (-m_{l+1},\,\dots,\,-m_r)\sum_{\substack{0\leq a_j <m_j\\\deg \alpha_j<d\\(\text{for }1\leq j \leq l)}}v^{a_1+\cdots+a_l} \alpha_1^{m_1-a_1}\cdots \alpha_l^{m_l-a_l}\cdot
\dbinom{m_1}{a_1}\cdots \\
&\,\cdots \dbinom{m_l}{a_l} \cdot \zeta_{\infty}^\star (-a_1,\,\dots,\,-a_l)\Biggr \} \\
&=\sum_{l=0}^{r}\Biggr\{\zeta_{< d}^\star (-m_{l+1},\,\dots,\,-m_r)\sum_{\substack{0\leq a_j <m_j\\f_j \in \mathbb{F}_q^\times \\(\text{for }1\leq j \leq l)}}v^{a_1+\cdots+a_l} \zeta_{<d}(a_1-m_1) \cdots \\ &\cdots \zeta_{<d}(a_l-m_l)\cdot f_1^{m_1-a_1}\cdots f_l^{m_l-a_l} \cdot
\dbinom{m_1}{a_1}\,\cdots \dbinom{m_l}{a_l} \cdot \zeta_{\infty}^\star (-a_1,\,\dots,\,-a_l)\Biggr \} \\
&=\sum_{l=0}^{r}\Biggr\{\zeta_{< d}^\star (-m_{l+1},\,\dots,\,-m_r)\sum_{\substack{0\leq a_j <m_j\\(q-1)|(m_j-a_j)\\(\text{for }1\leq j \leq l)}}(-1)^l v^{a_1+\cdots+a_l} \zeta_{<d}(a_1-m_1) \cdots \\
&\quad \ \cdots \zeta_{<d}(a_l-m_l)\cdot
\dbinom{m_1}{a_1}\cdots  \dbinom{m_l}{a_l} \cdot \zeta_{\infty}^\star(-a_1,\,\dots,\,-a_l)\Biggr \}.
\end{align}
\end{proof}

\begin{example}
(1). For $m \in \mathbb{Z}_{\geq0}$, we have
\begin{align}
    \zeta_{\infty}(-m)&=\zeta_{<d}(-m)-\sum_{\substack{0\leq a<m \\(q-1)|(m-a)}}v^{a}\zeta_{<d}(a-m)\binom{m}{a}\zeta_{\infty}(-a),\\
    \end{align}
(This appeared in \cite[Theorem 5.6]{Goss1979} and \cite[\S 5.3]{ThakurBook}.) 

(2). For $m_1,\,m_2\in \mathbb{Z}_{\geq0}$, we have
\begin{align}
    &\quad \ \zeta_\infty(-m_1,\,-m_2)\\
    &=\zeta_{<d}(-m_1,\,-m_2)-\zeta_{<d}(-m_2)\sum_{\substack{0\leq a_1<m_1 \\(q-1)|(m_1-a_1)}}v^{a_1}\zeta_{<d}(a_1-m_1)\binom{m_1}{a_1}\zeta_{\infty}(-a_1)\\
    &\quad \ +\sum_{\substack{0\leq a_1<m_1 \\0\leq a_2<m_2\\(q-1)|(m_1-a_1)\\(q-1)| (m_2-a_2)}}v^{a_1+a_2}\zeta_{<d}(a_1-m_1)\zeta_{<d}(a_2-m_2)\binom{m_1}{a_1}\binom{m_2}{a_2}\zeta_{\infty}(-a_1,\,-a_2),
\end{align}
and
\begin{align}
    &\quad \ \zeta_\infty^\star(-m_1,\,-m_2)\\
    &=\zeta_{<d}^\star(-m_1,\,-m_2)-\zeta_{<d}^\star(-m_2)\sum_{\substack{0\leq a_1<m_1 \\(q-1)|(m_1-a_1)}}v^{a_1}\zeta_{<d}(a_1-m_1)\binom{m_1}{a_1}\zeta_{\infty}(-a_1)\\
    &\quad \ +\sum_{\substack{0\leq a_1<m_1 \\0\leq a_2<m_2\\(q-1)|(m_1-a_1)\\(q-1)| (m_2-a_2)}}v^{a_1+a_2}\zeta_{<d}(a_1-m_1)\zeta_{<d}(a_2-m_2)\binom{m_1}{a_1}\binom{m_2}{a_2}\zeta_{\infty}^\star(-a_1,\,-a_2).
\end{align}

(3). For $m_1,\,\dots,\,m_{r-1} \in \mathbb{Z}_{\geq 0}$, we have 
\begin{align}
&\quad \ \zeta_\infty(-m_1,\,\dots,\,-m_{r-1},\,0) \label{depthdown}
\\&=\sum_{l=0}^{r-1}\Biggr\{\zeta_{<d}(-m_{l+1},\,\dots,\,-m_r)\sum_{\substack{0\leq a_j \leq m_j\\(q-1)|(m_j-a_j)\\(\text{for }1\leq j \leq l)}}(-1)^l v^{a_1+\cdots+a_l} \zeta_{<d}(a_1-m_1) \cdots\\
&\quad \cdots \zeta_{<d}(a_l-m_l)\cdot
\dbinom{m_1}{a_1}\cdots \dbinom{m_l}{a_l} \cdot \zeta_{\infty}(-a_1,\,\dots,\,-a_l) \Biggr \} , 
\end{align}
(here $l$ runs from $0$ to $r-1$, not to $r$).

\end{example}
\begin{remarks}\label{reminduction}
Since $S_i(0)=0$ for $i \geq 1$, we have $\zeta_\infty(-m_1,\,\dots,\,-m_r)=0$ if one of $m_1,\,\dots,\,m_{r-1}$ is zero for $m_1,\,\dots,\,m_r \in \mathbb{Z}_{\geq0}$. Combining this with $\eqref{depthdown}$, we can see that the special value $\zeta_\infty(-m_1,\,\dots,\,-m_r)$ can be written down $\zeta_\infty(-l_1,\,\dots,\,-l_{r^\prime})$'s with $r^\prime<r$ or $l_j<m_j$ for some $1\leq j\leq r$.
\end{remarks}
\subsection{The formulae for special values of $v$MZFs at negative integers}\label{4.2}
We show that special values of $v$MZFs at negative integers can be written in terms of those of $\infty$MZFs. Similar results for star versions is also shown, i.e. we write down special values of $v$MZSFs at negative integers in terms of those of $\infty$MZSFs.

\begin{theorem}\label{recursiveformulaev}
Let $m_1\,\dots,\,m_r \in \mathbb{N}_{\geq0}$ and let $v$ be a monic irreducible polynomial with degree $d$. Then we have the equalities
\begin{align}
& \quad \zeta_v(-m_1,\,\dots,\,-m_r)\label{vInductive} \\
&=\sum_{l=0}^{r}\Biggr\{\zeta_{<d}(-m_{l+1},\,\dots,\,-m_r)\sum_{\substack{0\leq a_j \leq m_j\\(q-1)|(m_j-a_j)\\(\text{for }1\leq j \leq l)}}(-1)^l v^{a_1+\cdots+a_l} \zeta_{<d}(a_1-m_1) \cdots\\
&\quad \cdots \zeta_{<d}(a_l-m_l)\cdot
\dbinom{m_1}{a_1}\cdots \dbinom{m_l}{a_l} \cdot \zeta_{\infty}(-a_1,\,\dots,\,-a_l) \Biggr \} . 
\end{align}
and
\begin{align}
& \qquad \zeta_v^\star(-m_1,\,\dots,\,-m_r)\\
&=\sum_{l=0}^{r}\Biggr\{\zeta_{< d}^\star (-m_{l+1},\,\dots,\,-m_r)\sum_{\substack{0\leq a_j \leq m_j\\(q-1)|(m_j-a_j)\\(\text{for }1\leq j \leq l)}}(-1)^l v^{a_1+\cdots+a_l} \cdot\zeta_{<d}(a_1-m_1) \cdots \\
&\quad  \cdots \zeta_{<d}(a_l-m_l)\cdot
\dbinom{m_1}{a_1}\cdots \dbinom{m_l}{a_l} \cdot \zeta_{\infty}^\star(-a_1,\,\dots,\,-a_l) \Biggr \}. \label{vstarInductive}
\end{align}
\end{theorem}
 Let us recall that these special values are in $k$ (Remark \ref{remarkink}). The difference between Theorems \ref{recursiveformulae} and \ref{recursiveformulaev} appears in the ranges which $a_j$'s run through. 
\begin{proof}
We have
\begin{align}
&\zeta_v(-m_1,\,\dots,\,m_l)
=\sum_{l=0}^{r} \left[\sum_{ i_1> \cdots>i_l \geq d > i_{l+1}>\cdots>i_r\geq0 } \left\{\sum_{\substack{\deg n_j= i_j\\n_j \in A_+\\ (n_j,\,v)=1\\(\text{for } r \geq j \geq 1)}}n_1^{m_1}\cdots n_r^{m_l}\right\}\right]\\
=&\sum_{l=0}^{r} \left[\zeta_{<d}(-m_{l+1},\,\dots,\,-m_r) \sum_{i_1>\cdots>i_l\geq d} \left\{\sum_{\substack{\deg n_j= i_j\\n_j \in A_+\\ (n_j,\,v)=1\\(\text{for }1\leq j \leq l)}}n_1^{m_1}\cdots n_l^{m_l}\right\}\right]\\
\intertext{here, let us recall that $d$ is the degree of $v$. Putting $n_j=vh_j+\alpha_j$, we have}
=&\sum_{l=0}^{r} \left[\zeta_{<d}(-m_{l+1},\dots,-m_r) \sum_{i_1>\cdots>i_l\geq d} \left\{\sum_{\substack{\deg h_j= i_j-d\\h_j \in A_+\\ \deg \alpha_j<d, \,\alpha_j \in A \setminus \{0\} \\(\text{for }1\leq j \leq l)}}(vh_1+\alpha_1)^{m_1}\cdots (vh_l+\alpha_l)^{m_l}\right\}\right]\\
\intertext{(The ranges which $\alpha_j$'s run through are different from those in the proof of Theorem \ref{recursiveformulae}.)}
=&\sum_{l=0}^{r} \Biggr[\zeta_{<d}(-m_{l+1},\,\dots,\,-m_r) \sum_{i_1>\cdots>i_l\geq d} \Biggr\{\sum_{\substack{\deg h_j= i_j-d\\h_j \in A_+\\ 0 \leq \deg \alpha_j<d \\(\text{for }1\leq j \leq l)}}
\Biggr(\sum_{\substack{0\leq a_j\leq m_j\\(\text{for }1\leq j \leq l)}} (vh_1)^{a_1}\cdots \\ &\cdots(vh_l)^{a_l}\cdot \alpha_1^{m_1-a_1}\cdots \alpha_l^{m_l-a_l}\cdot \dbinom{m_1}{a_1}\cdots \dbinom{m_l}{a_l}
\Biggr)\Biggr\}\Biggr]\\
=&\sum_{l=0}^{r}\Biggr\{\zeta_{<d}(-m_{l+1},\,\dots,\,-m_r)\sum_{\substack{0\leq a_j \leq m_j\\f_j \in \mathbb{F}_q^\times \\(\text{for }1\leq j \leq l)}}v^{a_1+\cdots+a_l} \zeta_{<d}(a_1-m_1) \cdots\zeta_{<d}(a_l-m_l) \\ 
& \cdot f_1^{m_1-a_1} \cdots f_l^{m_l-a_l} \cdot
\dbinom{m_1}{a_1}\,\cdots \dbinom{m_l}{a_l} \cdot \zeta_{\infty}(-a_1,\,\dots,\,-a_l)\Biggr \} \\
=&\sum_{l=0}^{r}\Biggr\{\zeta_{<d}(-m_{l+1},\,\dots,\,-m_r)\sum_{\substack{0\leq a_j  \leq m_j\\(q-1)|(m_j-a_j)\\(\text{for }1\leq j \leq l)}}(-1)^l v^{a_1+\cdots+a_l} \zeta_{<d}(a_1-m_1) \cdots  \\
&\qquad \qquad \cdots  \zeta_{<d}(a_l-m_l)\cdot \dbinom{m_1}{a_1}\cdots\dbinom{m_l}{a_l} \cdot \zeta_{\infty}(-a_1,\,\dots,\,-a_l)\Biggr \}.
\end{align}

The second formula can be confirmed as follows:
\begin{align}
&\zeta_v^\star(-m_1,\,\dots,\,m_l)
=\sum_{l=0}^{r} \left[\sum_{ i_1\geq \cdots \geq i_l \geq d > i_{l+1} \geq \cdots \geq i_r\geq0 } \left\{\sum_{\substack{\deg n_j= i_j\\n_j \in A_+\\ (n_j,\,v)=1\\(\text{for } r \geq j \geq 1)}}n_1^{m_1}\cdots n_r^{m_l}\right\}\right]\\
=&\sum_{l=0}^{r} \left[\zeta_{< d}^\star(-m_{l+1},\,\dots,\,-m_r) \sum_{i_1 \geq \cdots \geq i_l\geq d} \left\{\sum_{\substack{\deg n_j= i_j\\n_j \in A_+\\ (n_j,\,v)=1\\(\text{for }1\leq j \leq l)}}n_1^{m_1}\cdots n_l^{m_l}\right\}\right]\\
=&\sum_{l=0}^{r} \left[\zeta_{< d}^\star(-m_{l+1},\dots,-m_r) \sum_{i_1 \geq \cdots \geq i_l\geq d} \left\{\sum_{\substack{\deg h_j= i_j-d\\h_j \in A_+\\ \deg \alpha_j<d, \,\alpha_j \neq 0 \\(\text{for }1\leq j \leq l)}}(vh_1+\alpha_1)^{m_1}\cdots (vh_l+\alpha_l)^{m_l}\right\}\right]\\
=&\sum_{l=0}^{r}\Biggr\{\zeta_{< d}^\star(-m_{l+1},\,\dots,\,-m_r)\sum_{\substack{0\leq a_j \leq m_j\\f_j \in \mathbb{F}_q^\times \\(\text{for }1\leq j \leq l)}}v^{a_1+\cdots+a_l} \zeta_{<d}(a_1-m_1) \cdots \\ 
&\quad\cdots \zeta_{<d}(a_l-m_l)\cdot f_1^{m_1-a_1} \cdots f_l^{m_l-a_l}\cdot
\dbinom{m_1}{a_1}\,\cdots \dbinom{m_l}{a_l}\cdot \zeta_{\infty}^\star(-a_1,\,\dots,\,-a_l)\Biggr \} \\
=&\sum_{l=0}^{r}\Biggr\{\zeta_{< d}^\star(-m_{l+1},\,\dots,\,-m_r)\sum_{\substack{0\leq a_j  \leq m_j\\(q-1)|(m_j-a_j)\\(\text{for }1\leq j \leq l)}}(-1)^l v^{a_1+\cdots+a_l} \zeta_{<d}(a_1-m_1) \cdots  \\
&\qquad \cdots \zeta_{<d}(a_l-m_l)\cdot
\dbinom{m_1}{a_1}\cdots \dbinom{m_l}{a_l} \cdot \zeta_{\infty}^\star(-a_1,\,\dots,\,-a_l)\Biggr \}.
\end{align}
\end{proof}

\begin{example}
(1). For $m \in \mathbb{Z}_{\geq0}$, we have
\begin{align}
    &\quad\zeta_{v}(-m)=\zeta_{<d}(-m)-\sum_{\substack{0\leq a\leq m \\(q-1)|(m-a)}}v^{a}\zeta_{<d}(a-m)\binom{m}{a}\zeta_{\infty}(-a).
\end{align}

(2). For $m_1,\,m_2\in \mathbb{Z}_{\geq0}$, we have
\begin{align}
    &\quad \zeta_v(-m_1,\,-m_2)\\
    &=\zeta_{<d}(-m_1,\,-m_2)-\zeta_{<d}(-m_2)\sum_{\substack{0\leq a_1 \leq m_1 \\(q-1)|(m_1-a_1)}}v^{a_1}\zeta_{<d}(a_1-m_1)\binom{m_1}{a_1}\zeta_{\infty}(-a_1)\\
    &\quad +\sum_{\substack{0\leq a_1 \leq m_1 \\0\leq a_2 \leq m_2\\(q-1)|(m_1-a_1)\\(q-1)| (m_2-a_2)}}v^{a_1+a_2}\zeta_{<d}(a_1-m_1)\zeta_{<d}(a_2-m_2)\binom{m_1}{a_1}\binom{m_2}{a_2}\zeta_{\infty}(-a_1,\,-a_2),
\end{align}
and
\begin{align}
    &\quad \zeta_v^\star(-m_1,\,-m_2)\\
    &=\zeta_{< d}^\star(-m_1,\,-m_2)-\zeta_{< d}^\star(-m_2)\sum_{\substack{0\leq a_1 \leq m_1 \\(q-1)|(m_1-a_1)}}v^{a_1}\zeta_{<d}(a_1-m_1)\binom{m_1}{a_1}\zeta_{\infty}(-a_1)\\
    &\quad +\sum_{\substack{0\leq a_1 \leq m_1 \\0\leq a_2 \leq m_2\\(q-1)|(m_1-a_1)\\(q-1)| (m_2-a_2)}}v^{a_1+a_2}\zeta_{<d}(a_1-m_1)\zeta_{<d}(a_2-m_2)\binom{m_1}{a_1}\binom{m_2}{a_2}\zeta_{\infty}^\star(-a_1,\,-a_2).
\end{align}
\end{example}

\subsection{Relationship between $\infty$MZFs and $v$MZFs} \label{4.3}

In the case of characteristic $0$, there is a formula connecting ``special values'' of multiple zeta functions at negative integers with those of $p$-adic $L$-functions shown in  \cite[Remark 2.2]{Furusho2017b}.
Using the results of previous subsections, we obtain similar relationship between special values of $\infty$MZFs and $v$MZFs at negative integers.

\begin{theorem}\label{recursiveformulaeCor}
Special values of $v$MZFs at negative integers can be written in terms of those of $\infty$MZFs by the following identity which holds for $m_1\,\dots,\,m_r \in \mathbb{N}_{\geq0}$:

\begin{align}
\quad &\quad \zeta_v(-m_1,\,\dots,\,-m_r)-\zeta_{\infty}(-m_1,\,\dots,\,-m_r) \label{EqInductiveCor}\\
&=\sum_{l=0}^{r}\Biggr\{\zeta_{<d}(-m_{l+1},\,\dots,\,-m_r)\sum_{\substack{0\leq a_j \leq m_j\\(q-1)|(m_j-a_j)\\(1\leq j \leq l)\\a_j=m_j \text{ for some $j$}}}(-1)^l v^{a_1+\cdots+a_l} \zeta_{<d}(a_1-m_1) \cdots\\
&\ \cdots \zeta_{<d}(a_l-m_l)\cdot
\dbinom{m_1}{a_1}\cdots \dbinom{m_l}{a_l}\cdot\zeta_{\infty}(-a_1,\,\dots,\,-a_l) \Biggr \} . 
\end{align}
Conversely, special values of $\infty$MZFs at negative integers can be written in terms of those of $v$MZFs.
\end{theorem}
\begin{proof}
Comparing the equations \eqref{inftyInductive} and \eqref{vInductive}, we obtain \eqref{EqInductiveCor}. As \eqref{EqInductiveCor} implies 
\begin{equation}
    \zeta_v(-m)=(1-v^{m})\zeta_{\infty}(-m), \ (m \in \mathbb{N}), \label{Dep1interpolate}
\end{equation}
we can write down special values of $\infty$MZFs at negative integers in terms of those of $v$MZFs if $r=1$, and we can reduced to the case $r=1$ by Remark \ref{reminduction}.
\end{proof}
By virtue of Theorem \ref{recursiveformulaeCor}, we can write down special values of $\infty$MZFs at negative integers in terms of those of $v$MZFs, and vice versa.
\begin{example}
In the case of $r=1$, the corollary above yields the well-known formula \eqref{Dep1interpolate}.
 If $r=2$, we can write down special values of the $v$MZF in terms of special values of $\infty$MZFs as follows:
 \begin{align}
 \quad \ \zeta_v(-m_1,\,-m_2)&=(1+v^{m_1+m_2})\zeta_\infty(-m_1,\,-m_2)-\zeta_{<d}(-m_2)v^{m_1}\zeta_\infty(-m_1)  \label{dep2ex1}\\
 &\, +\sum_{\substack{0\leq a_1 <m_1\\(q-1)|(m_1-a_1)}}v^{a_1+m_2}\zeta_{<d}(a_1-m_1)\dbinom{m_1}{a_1}\zeta_\infty(-a_1,\, -m_2)\\
 &\, +\sum_{\substack{0\leq a_2 < m_2\\(q-1)|(m_2-a_2)}}v^{m_1+a_2}\zeta_{<d}(a_2-m_2)\dbinom{m_2}{a_2}\zeta_\infty(-m_1,\, -a_2).
 \end{align}

Conversely, recursive applications of \eqref{dep2ex1} induces the following formula, which expresses special values of $\infty$MZFs at negative integers in terms of those of $v$MZFs, for $m_1,\,m_1^\prime \in \mathbb{Z}_{\geq0}$,
\begin{align*}
     & \quad \ v^{m_1+m_1^\prime}\zeta_\infty(-m_1,\,-m_1^\prime)\\
     &=\sum_{\substack{l,\,l^\prime \geq 1\\m_1>m_2>\cdots\,>m_l\geq 0\\m_1^\prime>m_2^\prime>\cdots\,>m_{l^\prime}^\prime \geq0 \\ (q-1)|(m_{i-1}-m_i)\\(q-1)|(m_{i-1}^\prime-m_i^\prime)}}
     \left\{  \zeta_v(-m_l,\,-m_{l^\prime}^\prime)+ \zeta_{<d}(-m_{l^\prime}^\prime) \frac{v^{m_l}}{1-v^{m_l}} \zeta_v(-m_l)\right\}\\
    &\quad \cdot \Biggl[ -\sum_{ [\dots,(j_i,\,j_i^\prime),\dots]  \in \operatorname{S}_{l,\,l^\prime}} \prod_{i=1}^{l+l^\prime-1}\Biggl\{\zeta_{<d}(m_{j_{i-1}}-m_{j_i})\zeta_{<d}(m_{j_{i-1}^\prime}^\prime-m_{j_i^\prime}^\prime) \\ &\quad \quad \cdot\frac{-v^{m_{j_i}+m_{{j_i^\prime}}^\prime}}{1+v^{m_{j_i}+m_{{j_i^\prime}}^\prime}}\dbinom{m_{j_{i-1}}}{m_{j_i}}\dbinom{m_{{j_{i-1}^\prime}}^\prime}{m_{{j_i^\prime}}^\prime}\  \Biggr\} \Biggr]\\
    &=-\sum_{\substack{l,\,l^\prime \geq 1\\ [\dots,(j_i,\,j_i^\prime),\dots]  \in \operatorname{S}_{l,\,l^\prime}}}\sum_{\substack{m_1>m_2>\cdots\,>m_l\geq 0\\m_1^\prime>m_2^\prime>\cdots\,>m_{l^\prime}^\prime \geq0 \\ (q-1)|(m_{i-1}-m_i),(m_{i-1}^\prime-m_i^\prime)}}
    \Biggl[  \prod_{i=1}^{l+l^\prime-1}\Biggl\{\zeta_{<d}(m_{j_{i-1}}-m_{j_i}) \\ 
    &\quad \quad  \cdot\zeta_{<d}(m_{j_{i-1}^\prime}^\prime-m_{j_i^\prime}^\prime)\frac{-v^{m_{j_i}+m_{{j_i^\prime}}^\prime}}{1+v^{m_{j_i}+m_{{j_i^\prime}}^\prime}}\dbinom{m_{j_{i-1}}}{m_{j_i}}\dbinom{m_{{j_{i-1}^\prime}}^\prime}{m_{{j_i^\prime}}^\prime}\  \Biggr\} \Biggr]\\
    & \quad \ \cdot \left\{  \zeta_v(-m_l,\,-m_{l^\prime}^\prime)+ \zeta_{<d}(-m_{l^\prime}^\prime) \frac{v^{m_l}}{1-v^{m_l}} \zeta_v(-m_l)\right\}.\\
\end{align*}
Here, we define $\zeta_{<d}(m_{j_0}-m_{j_1})=\zeta_{<d}(m_{j_0}^\prime- m_{j_1}^\prime)=\binom{m_{j_0}}{m_{j_1}}=\binom{m_{j_0}^\prime}{m_{j_1}^\prime} \coloneqq 1$ and define $\operatorname{S}_{l,\,l^\prime}$ to be the set of maximal strictly increasing sequences in $\{1,\,\dots,\,l\} \cdot \{1,\,\dots,\,l ^\prime\}$. In other words, $\operatorname{S}_{l,\,l^\prime}$ is the  set of sequences $[(j_1,\,j_1^\prime)$, $\dots$,$(j_{l+l^\prime-1}$, $j_{l+l^\prime-1}^\prime)]$ of length $l+l^\prime-1$ such that $j_1=j_1^\prime=1$, $(j_i,\,j_i^\prime)\neq (j_{i+1},\,j_{i+1}^\prime)$, $j_{i+1}=j_i,\,j_i+1$ and $j_{i+1}^\prime=j_i^\prime,\,j_i^\prime+1$ for all $1 \leq i \leq l+l^\prime-2$ (these conditions imply  that $j_{l+l^\prime-1}=l$ and $j_{l+l^\prime-1}^\prime=l^\prime$). For example, $[(1,1),\,(1,2),\,(1,3),\,(2,3),\,\dots,\,(l-1,l^\prime),\,(l,l^\prime)] \in \operatorname{S}_{l,\,l^\prime}$.
\end{example}

By the similar arguments, we obain the following formulae for star-versions:
\begin{theorem}\label{recursiveformulaeCorstar}
For $m_1\,\dots,\,m_r \in \mathbb{N}_{\geq0}$, we have the following identity which enable us to write down special values of $v$MZSFs at negative integers in terms of those of $\infty$MZSFs:

\begin{align}
\quad &\quad \zeta_v^\star(-m_1,\,\dots,\,-m_r)-\zeta_{\infty}^\star(-m_1,\,\dots,\,-m_r) \label{EqInductiveCorstar}\\
&=\sum_{l=0}^{r}\Biggr\{\zeta_{<d}^\star(-m_{l+1},\,\dots,\,-m_r)\sum_{\substack{0\leq a_j \leq m_j\\(q-1)|(m_j-a_j)\\(1\leq j \leq l)\\a_j=m_j \text{ for some $j$}}}(-1)^l v^{a_1+\cdots+a_l} \zeta_{<d}(a_1-m_1) \cdots\\
&\ \cdots \zeta_{<d}(a_l-m_l)\cdot
\dbinom{m_1}{a_1}\cdots \dbinom{m_l}{a_l}\cdot\zeta_{\infty}^\star(-a_1,\,\dots,\,-a_l) \Biggr \} . 
\end{align}
Conversely, we can write down special values of $\infty$MZSFs at negative integers in terms of those of $v$MZSFs.
\end{theorem}
Theorem \ref{recursiveformulaeCorstar} enable us to write down special values of $\infty$MZSFs at negative integers in terms of those of $v$MZSFs, and vice versa.
\vskip\baselineskip

\textbf{Acknowledgement}
The author is deeply grateful to his doctoral advisor, Professor H. Furusho for his profound instruction and continuous encouragements. He also grateful to Proffessors R.~Harada, Y.~Mishiba and T.~Ngo~Dac who gave very variable comments on manuscript of this paper.

\appendix 
\section{Analytic properties of $\infty$MZFs and $v$MZFs}\label{AppA}
In this appendix, we show the convergence of infinite series defining $\infty$MZFs and $v$MZFs and rigid analyticity of these functions by reproducing the arguments of Goss in \cite{Goss1983} in order to make present paper self-contained.

\begin{lemma}\label{lemmaestimationW}
Let ${\operatorname{ord}}$ be $\operatorname{ord}_\infty$ or $\operatorname{ord}_v$ and let $L/k$ be a field extension to which the valuation $\operatorname{ord}$ is prolonged. Let us take finite dimensional $\mathbb{F}_q$-vector subspace $W$ of $L$ such that ${\operatorname{ord}}(w)>0$ for all $w \in W$ and let
\begin{equation}
    W_j\coloneqq \{w \in W \mid {\operatorname{ord}}(w)\geq j\}.
\end{equation}
If we take $w_N \in k_{\operatorname{ord}}$ such that ${\operatorname{ord}}(w_N)>0$, then
\begin{equation}
    {\operatorname{ord}}\left(\sum_{w \in W}(1+w_N+w)^t \right) \geq (q-1) \sum_{j\geq1}\dim W_j
\end{equation}
for every $t\in \mathbb{Z}_p$ if ${\operatorname{ord}}=\operatorname{ord}_\infty$ and $t \in {\mathbb{Z}_{d,\,p}}$ if ${\operatorname{ord}}=\operatorname{ord}_v$.
\end{lemma}

\begin{proof}
We reproduce the proof of this lemma which can be observed in \cite[Corollary 1.3 and the proof of Theorem 2.2.3]{Goss1983}.
As the set $\{ a \in k_{\operatorname{ord}} \mid {\operatorname{ord}}(a) \geq m\}$ is closed for any $m\in\mathbb{Z}$, it is enough to establish the lemma with $t \in \mathbb{N}$. We have 
\begin{align}
    {\operatorname{ord}}\left(\sum_{w \in W}(1+w_N+w)^t\right)&={\operatorname{ord}}\left\{\sum_{m=0}^t\dbinom{t}{m} w_N^{t-m}\middle(\sum_{w\in W}(1+w)^m\middle)\right\}\\
    &={\operatorname{ord}}\left\{\sum_{m=0}^t  \dbinom{t}{m} w_N^{t-m} \sum_{l=0}^m \dbinom{m}{l} \middle(\sum_{w\in W}w^l\middle)\right\}\\
    &\geq \min_{0\leq l \leq m\leq t} {\operatorname{ord}}\left\{ \dbinom{t}{m} \dbinom{m}{l} w_N^{t-m}\sum_{w \in W} w^l \right\} \\
    &\geq (q-1) \sum_{j\geq1}\dim W_j.
\end{align}
Here, the last inequality follows from \cite[Lemma 1.2]{Goss1983}.
\end{proof}

The coefficient of $\sigma_1^{-i_1}\cdots\sigma_r^{-i_r}$ in the series \eqref{EqDefMZF} and \eqref{EqDefMZSF} is
\begin{equation}
\left(\sum_{\substack{\deg n=i_1\\n \in A_+}}(\theta^{-i_1}n)^{-t_1}\right)\cdots\left(\sum_{\substack{\deg n=i_r\\n \in A_+}}(\theta^{-i_r}n)^{-t_r}\right). \label{coef}
\end{equation}
As the negative of order
\begin{equation}
     -\operatorname{ord}_\infty \left (\sigma_1^{i_1}\cdots\sigma_r^{i_r} \right )
\end{equation}
increases at most linearly as each $i_j$ tends to $\infty$ for fixed $\sigma_1\,\dots,\,\sigma_r$, it is enough to show that the order of the coefficients \eqref{coef} grows quadratically as each $i_j$ tends to $\infty$ ($1\leq j \leq r$) in order to confirm the convergence of the sums  \eqref{EqDefMZF} and \eqref{EqDefMZSF} in Definition \ref{DefMZF}. 

\begin{prop}\label{Appinfty}
The order
\begin{equation}
\operatorname{ord}_\infty \left(\sum_{\substack{\deg n=i\\n \in A_+}}(\theta^{-i}n)^{-t}\right),\quad (i\in \mathbb{N},\,t \in {\mathbb{Z}_{p}})
\end{equation}
grows quadratically as $i$ tends to $\infty$.
\end{prop}
\begin{proof}
By using Lemma \ref{lemmaestimationW} with $\operatorname{ord}=\operatorname{ord}_\infty$ and $W\coloneqq  \langle \theta^{-1},\,\dots,\,\theta^{-i} \rangle_{\mathbb{F}_q} \subset k_\infty$, we have, for $i \in \mathbb{N}$ and $t \in \mathbb{Z}_p$,
\begin{equation}
    \operatorname{ord}_\infty \left(\sum_{\substack{\deg n=i\\n \in A_+}}(\theta^{-i}n)^{-t}\right)=\operatorname{ord}_\infty \left(\sum_{w \in W}(1+w)^{-t}\right)\geq (q-1)\sum_{j\geq1} \dim W_j.
\end{equation}
Since
\begin{equation}
    W_j=\begin{cases}
    W &j=0,\\
    \langle \theta^{-j},\dots,\,\theta^{-i} \rangle_{\mathbb{F}_q} & 1\leq j \leq i,\\
    0 &j >i,
    \end{cases} 
\end{equation}
we have
\begin{equation}
    \operatorname{ord}_\infty \left ( \sum_{\substack{\deg n=i\\n \in A_+}}(\theta^{-i}n)^t \right ) \geq\frac{i(i+1)}{2}.
\end{equation}
\end{proof}
Consequently, we have the following result:
\begin{theorem}[{\cite[Corollary 6.1]{Angles2016}}]
The sums \eqref{EqDefMZF} and \eqref{EqDefMZSF} in Definition \ref{DefMZF} converge. Therefore, for each $t_1,\,\dots,\,t_r \in X_\infty$, the functions
\begin{equation}
    \zeta_{\infty}((\sigma_1,\,t_1),\,\dots,\,(\sigma_r,\,t_r)) \text{ and } \zeta_{\infty}^\star((\sigma_1,\,t_1),\,\dots,\,(\sigma_r,\,t_r))
\end{equation}
are rigid analytic, that is, they are represented by power series in $\sigma_1,\,\dots,\,\sigma_r$ of infinite radii of convergence.
\end{theorem}
This was shown in \cite[Corollary 6.1]{Angles2016} in more generalized setting.

Similar to the $\infty$-adic case, in order to show that the sum \eqref{EqDef$v$MZF} and \eqref{EqDefvMZSF} converge, it is enough to confirm the quadratic growth of the valuations:
\begin{equation}
    \operatorname{ord}_v\left(\tilde{S}_{i_1}(t_1)\right),\,\dots, \,\operatorname{ord}_v \left(\tilde{S}_{i_r}(t_r)\right) \quad \  (i_1,\,\dots,\,i_r \to \infty).
\end{equation}

\begin{prop}\label{Lemmaestimation}
The valuation
\begin{equation}
    \operatorname{ord}_v\left(\tilde{S}_i(t)\right) \quad (i\in \mathbb{N},\,t \in {\mathbb{Z}_{d,\,p}})
\end{equation}
increase quadratically as $i$ tends to $\infty$.
\end{prop}

\begin{proof}
We reproduce the proof of this proposition which can be observed in \cite[Theorem 5.4.5]{Goss1983}. It is enough to estimate
\begin{equation}
    \operatorname{ord}_v\left(\sum_{n \in X_{i,\,f}}n^{-t}\right)
\end{equation}
where $X_{i,\,f}\coloneqq \{n \in A_+ \mid \deg n =i,\,n \equiv f \mod v\}$, for a fixed $f\in A$ with $0\leq \deg f< d$ and for large $i$. We can write $n \in X_{i,\,f}$ as $n=\omega(n)\langle n \rangle_v$ where $\omega(n)\equiv n \mod v$ and $\langle n \rangle_v$ is the $1$-unit. Then, we have $\omega(n)=\omega(f) \in A_v$. Let us put $w_n\coloneqq \langle n \rangle_v-1$ for $n \in X_{i,\,f}$ and fix $N \in X_{i,\,f}$ such that $\max_{n\in X_{i,\,f}}\operatorname{ord}_v(n-f)=\operatorname{ord}_v(N-f)$.  Then the sets $Y\coloneqq \{ N-n \mid n\in X_{i,\,f}\}$ and $W\coloneqq  \{w_N-w_n \mid n \in X_{i,\,f}\}$ are $\mathbb{F}_q$-linear subspace of $k_v$ and the map $m:(N-n) \mapsto -\omega(f)^{-1}(N-n)=(w_N-w_n)$ is an $\mathbb{F}_q$-linear isomorphism preserving valuation $\operatorname{ord}_v$ from $Y$ to $W$.
Indeed, as the set $Y$ is obviously the $\mathbb{F}_q$-linear subspace and the element $-\omega(f)^{-1}$ is invertible with $\operatorname{ord}_v(-\omega(f)^{-1})=1$, we can see that the set $W$ is also finite dimensional $\mathbb{F}_q$-linear subspace of $k_v$ and $m$ is a valuation preserving $\mathbb{F}_q$-linear isomorphism.

Putting $W_j\coloneqq \{w \in W \mid \operatorname{ord}_v \geq j \}$ and $Y_j\coloneqq \{N-n\mid n\in X_{i,\,f},\, \operatorname{ord}_v(N-n)\geq j\}$ as in Lemma \ref{lemmaestimationW}, we have

\begin{align}
     &\quad \, \operatorname{ord}_v\left(\sum_{n \in X_{i,\,f}}n^{-t}\right)=\operatorname{ord}_v\left(\omega(f)^{-t} \sum_{n \in X_{i,\,f}}(1+w_n )^{-t} \right)\\
     &=\operatorname{ord}_v\left(\omega(f)^{-t} \sum_{w \in W}(1+w_N+w)^{-t} \right)\\
     &\geq \operatorname{ord}_v\left(\omega(f)^{-t} \right)(q-1)\sum_{j\geq0}\dim W_j=v\left(\omega(f)^{-t} \right)(q-1)\sum_{j\geq1}\dim Y_j
\end{align}
by Lemma \ref{lemmaestimationW}.
Since $\max_{n\in X_{i,\,f}}\operatorname{ord}_v(n-f)=\operatorname{ord}_v(N-f)$, we obtain
\begin{align}
    \#Y_j&=\#\{n\in X_{i,\,f} \mid \operatorname{ord}_v(N-n) \geq j\}\\
    &=\#\{n\in X_{i,\,f} \mid \operatorname{ord}_v(n-f) \geq j\}\\
    &=\# \{v^j g+f \mid g\in A_+ ,\,\deg g = i-jd\}=q^{i-jd-1},
\end{align}
here, the second equality is obtained as follows:
\begin{equation}    
\operatorname{ord}_v(n-f)=\min\{\operatorname{ord}_v(N-f), \operatorname{ord}_v(n-f)\} = \operatorname{ord}_v(N-n).
\end{equation}
Thus, the inequality
\begin{equation}
    \operatorname{ord}_v\left(\sum_{n \in X_{i,\,f}}n^{-t}\right)
    \geq (q-1)\sum_{j=1}^{\left[\frac{i}{d}\right]}(i-jd-1).
\end{equation}
holds and the right hand side increase quadratically as $i$ tends to $\infty$, hence the proposition follows.
\end{proof}
Consequently, we have the following result, which was shown in \cite[Corollary 6.2]{Angles2016} in more generalized setting.
\begin{theorem}[{\cite[Corollary 6.2]{Angles2016}}]
The series \eqref{EqDef$v$MZF} and \eqref{EqDefvMZSF} in Definition \ref{Def$v$MZF} converge. Consequently, for each $t_1,\,\dots,\,t_r \in {\mathbb{Z}_{d,\,p}}$, the functions
\begin{equation}
    \zeta_v((\sigma_1,\,t_1),\,\dots,\,(\sigma_r,\,t_r)) \text{ and } \zeta_v^\star((\sigma_1,\,t_1),\,\dots,\,(\sigma_r,\,t_r))
\end{equation}
are rigid analytic in $\sigma_1^{-1},\,\dots,\,\sigma_r^{-1}$, that is, they are represented by power series  of infinite radii of convergence.
\end{theorem}

%\bibliographystyle{siam}    
%\bibliography{C:/Users/daich/OneDrive/Cloud/BibTexLibraries/BibTexLibrary1.bib} 
%\bibliography{/Users/MatsuzukiDaichi/OneDrive/Cloud/BibTexLibraries/BibTexLibrary1.bib} 

%\bibliography{BibTexLibrary1.bib} 

\end{document}